\documentclass[oneside,11pt,reqno]{amsart}
\usepackage{amssymb,amsmath,amsthm,bbm,enumerate,mdwlist,url,multirow,mathtools,hyperref}
\usepackage[pdftex]{graphicx}
\usepackage[shortlabels]{enumitem}

\theoremstyle{definition}
\newtheorem{definition}{Definition}
\theoremstyle{theorem}
\newtheorem{proposition}[definition]{Proposition}
\newtheorem{lemma}[definition]{Lemma}
\newtheorem{theorem}[definition]{Theorem}
\newtheorem{assumption}[definition]{Assumption}
\newtheorem{corollary}[definition]{Corollary}
\numberwithin{equation}{section}
\numberwithin{definition}{section}
\theoremstyle{remark}
\newtheorem{remark}[definition]{Remark}
\newtheorem{example}[definition]{Example}
\makeatletter
\newcommand{\mytag}[2]{%
  \text{#1}%
  \@bsphack
  \begingroup
    \@onelevel@sanitize\@currentlabelname
    \edef\@currentlabelname{%
      \expandafter\strip@period\@currentlabelname\relax.\relax\@@@%
    }%
    \protected@write\@auxout{}{%
      \string\newlabel{#2}{%
        {#1}%
        {\thepage}%
        {\@currentlabelname}%
        {\@currentHref}{}%
      }%
    }%
  \endgroup
  \@esphack
}
\makeatother
\newcounter{term}
\renewcommand*{\theterm}{(\Alph{term})}

\AtBeginDocument{%
  \let\mylabel\label
}

\newcommand{\mytagg}[1]{%
  \begingroup 
    \refstepcounter{term}%
    \mylabel{#1}%
    \text{\theterm}%
  \endgroup
}

\def\PP{\mathsf P}
\def\EE{\mathsf E}

\def\Zh{\mathbb{Z}_h}
 \def\conv{\mathrm{conv}}

\def\ind{\mathbbm{1}_{[-V,V]}}
\def\sgn{\mathrm{sgn}}
\def\indd{\mathbbm{1}_{[-1,1]}}

\def\measure{\lambda}
\def\diffusion{\sigma^2}
\def\drift{\mu}

\def\Ind{\mathbbm{1}_{[-V,V]^d}}
\def\Indd{\mathbbm{1}_{[-1,1]^d}}

\def\Measure{\lambda}
\def\Diffusion{\Sigma}
\def\Drift{\mu}
\def\hatdiffusion{\hat{\sigma}^2}
\begin{document}
\title{Markov chain approximations for transition densities of L\'evy processes}

\author{Aleksandar Mijatovi\'{c}}
\address{Department of Mathematics, Imperial College London, UK}
\email{a.mijatovic@imperial.ac.uk}

\author{Matija Vidmar}
\address{Department of Statistics, University of Warwick, UK}
\email{m.vidmar@warwick.ac.uk}

\author{Saul Jacka}
\address{Department of Statistics, University of Warwick, UK}
\email{s.d.jacka@warwick.ac.uk}

\thanks{MV acknowledges the support of the Slovene Human Resources Development and Scholarship Fund under contract number 11010-543/2011.}

\begin{abstract}
We consider the convergence of a continuous-time Markov chain approximation $X^h$, $h>0$, to an $\mathbb{R}^d$-valued  L\'evy process $X$. The state space of $X^h$ is an equidistant lattice and its $Q$-matrix is chosen to approximate the generator of $X$. In dimension one ($d=1$), and then under a general sufficient condition for the existence of transition densities of $X$, we establish sharp convergence rates of the normalised probability mass function of $X^h$ to the probability density function of $X$. In higher dimensions ($d>1$), rates of convergence are obtained under a technical condition, which is satisfied when the diffusion matrix is non-degenerate. 
\end{abstract}



\maketitle

\allowdisplaybreaks

\section{Introduction}\label{section:Introduction}
Discretization schemes for stochastic processes are relevant both theoretically, as they shed light on the nature of the underlying stochasticity, and practically, since they lend themselves well to numerical methods. L\'evy processes, in particular, constitute a rich and fundamental class with applications in diverse areas such as mathematical finance, risk management, insurance, queuing, storage and population genetics etc. (see e.g.~\cite{kyprianou}). 

\subsection{Short statement of problem and results}\label{subsection:short_statement}
In the present paper, we study the rate of convergence of a weak approximation of an $\mathbb{R}^d$-valued ($d\in \mathbb{N}$) L\'evy process $X$ by a continuous-time Markov chain (CTMC). Our main aim is to understand the rates of convergence of transition densities. These cannot be viewed as expectations of (sufficiently well-behaved, e.g. bounded continuous) real-valued functions against the marginals of the processes, and hence are in general hard to study. 

Since the results are easier to describe in dimension one ($d=1$), we focus first on this setting. Specifically, our main result in this case, Theorem~\ref{theorem:convergencerates}, establishes the precise convergence rate of the normalised  probability mass function of the approximating Markov chain to the transition density of the L\'evy process for the two proposed discretisation schemes, one in the case where $X$ has a non-trivial diffusion component and one when it does not. More precisely, in both cases we approximate $X$ by a CTMC $X^h$  with state space $\Zh:=h\mathbb{Z}$ and $Q$-matrix defined as a natural discretised version of the generator of $X$. This makes the CTMC $X^h$ into a continuous-time random walk, which is skip-free (i.e. simple) if $X$ is without jumps (i.e.  Brownian motion with drift). The quantity $$\kappa(h):=\int_{[-1,1] \backslash [-h,h]}\vert x\vert d\lambda(x),$$ where $\lambda$ is the L\'evy measure of $X$, is related to the activity of the small jumps of $X$ and plays a crucial role in the rate of convergence.  We assume that either the diffusion component of $X$ is present ($\sigma^2>0$) or the jump activity of $X$ is sufficient (Orey's condition \cite[p. 190, Proposition 28.3]{sato}, see also Assumption~\ref{assumption} below) to ensure that $X$ admits continuous transition densities $p_{t,T}(x,y)$ (from $x$ at time $t$ to $y$ at time $T>t$), which are our main object of study. 

Let $P^h_{t,T}(x,y):=\PP(X^h_T=y\vert X^h_t=x)$ denote the corresponding transition probabilities of $X^h$ and let $$\Delta_{T-t}(h):=\sup_{x,y\in \Zh}\left\vert p_{t,T}(x,y)-\frac{1}{h}P^h_{t,T}(x,y)\right\vert.$$ The following table summarizes our result (for functions $f\geq 0$ and $g>0$, we shall write $f=O(g)$ (resp. $f=o(g)$, $f\sim g$) for $\limsup_ {h\downarrow 0}f(h)/g(h)<\infty$  (resp. $\lim_{h\downarrow 0}f(h)/g(h)=0$, $\lim_{h\downarrow 0}f(h)/g(h)\in (0,\infty)$) --- if $g$ converges to $0$, then we will say $f$ decays no slower than (resp. faster than, at the same rate as) $g$):
\begin{center}
\begin{tabular}{|c|c|c|}\hline
 & $\sigma^2>0$ & $\sigma^2=0$\\\hline
$\lambda(\mathbb{R})=0$ &\hspace{0.05cm} $\Delta_{T-t}(h)=O(h^2)$ &\phantom{spacee} $\times$\phantom{spacee}\\\hline
$0<\lambda(\mathbb{R})<\infty$ & $\Delta_{T-t}(h)=O(h)$ & \phantom{spacee}\hspace{0.09cm}$\times$\phantom{spacee}\\\hline
$\lambda(\mathbb{R})=\infty$  & \multicolumn{2}{|c|}{$\Delta_{T-t}(h)=O(h\kappa(h/2))$}\\\hline
\end{tabular}
\end{center}
We also prove that the rates stated here are sharp in the sense that there exist L\'evy processes for which convergence is no better than stated. 

Note that the rate of convergence depends on the L\'evy measure $\lambda$, it being best when $\measure=0$ (quadratic when $\diffusion>0$), and linear otherwise, unless the pure jump part of $X$ has infinite variation, in which case it depends on the quantity $\kappa$. This is due to the nature of the discretisation of the Brownian motion with drift (which gives a quadratic order of convergence, when $\diffusion>0$), and then of the L\'evy measure, which is aggregated over intervals of length $h$ around each of the lattice points; see also \ref{remark:characteristic_exponents:iv} of Remark~\ref{remark:characteristic_exponents}. In the infinite activity case, $\kappa(h)=o(1/h)$, indeed $\kappa$ is bounded, if in addition $\kappa(0)<\infty$. However, the convergence of $h\kappa(h/2)$ to zero, as $h\downarrow0$, can be arbitrarily slow. Finally, if $X$ is a compound Poisson process (i.e. $\lambda(\mathbb R)\in(0,\infty)$) without a diffusion component, but possibly with a drift, there is always an atom present in the law of $X$ at a fixed time, which is why the finite L\'evy measure case is studied only when $\sigma^2>0$.

The proof of Theorem~\ref{theorem:convergencerates} is in two steps: we first establish the convergence rate of the characteristic exponent of $X_t^h$ to that of $X_t$  (Subsection~\ref{exponentials}). In the second step we apply this to the study of the convergence of transition densities (Section~\ref{section:Rates}) via their spectral representations (established in Subsection~\ref{subsection:integralrep}).  Note that in general the rates of convergence of the characteristic functions do not carry over directly to the distribution functions. We are able to follow  through the above programme by exploiting the special structure of the  infinitely divisible distributions in what amounts to a detailed comparison of the transition kernels  $p_{t,T }(x, y)$ and $P^h_{t,T}(x, y)$.

By way of example, note that if $\lambda([-1,1]\backslash [-h,h])\sim 1/h^{1+\alpha}$ for some $\alpha\in (0,1)$, then $\kappa(h)\sim h^{-\alpha}$ and the convergence of the normalized  probability mass function to the transition density is by Theorem~\ref{theorem:convergencerates} of order $h^{1-\alpha}$, since $\kappa(0)=\infty$ and Orey's condition is satisfied. In particular, in the case of the CGMY~\cite{madan} (tempered stable) or $\beta$-stable~\cite[p. 80]{sato} processes with stability parameter $\beta\in(1,2)$, we have $\alpha=\beta-1$ and hence convergence of order $h^{2-\beta}$. More generally, if $\beta:=\inf\{p>0:\int_{[-1,1]}\vert x\vert^pd\measure(x)<\infty\}$ is the Blumenthal-Getoor index, and $\beta\geq 1$, then for any $p>\beta$, $\zeta(\delta)=O(\delta^{2-p})$. Conversely, if for some $p\geq 1$, $\zeta(\delta)=O(\delta^{2-p})$, then $\beta\leq p$. 

This gives the overall picture in dimension one. In dimensions higher than one ($d>1$), and then under a straightforward extension of the discretization described above, essentially the same rates of convergence are obtained as in the univariate case; this time under a technical condition (cf. Assumption~\ref{assumption:multivariate}), which is satisfied when the diffusion-matrix is non-degenerate. Our main result in this case is Theorem~\ref{theorem:convergence:multivariate}. 

\subsection{Literature overview}\label{subsection:literature_review}
In general, there has been a plethora of publications devoted to the subject of discretization schemes for stochastic processes, see e.g. \cite{kloeden}, and with regard to the pricing of financial derivatives \cite{glasserman} and the references therein. In particular, there exists a wealth of literature concerning approximations of L\'evy processes in one form or another and a brief overview of simulation techniques is given by \cite{rosinski}. 

In continuous time, for example, \cite{kiessling} approximates by replacing the small jumps part with a diffusion, and discusses also rates of convergence for $\EE[g\circ X_T]$, where $g$ is real-valued and satisfies certain integrability conditions, $T$ is a fixed time and $X$ the process under approximation; \cite{crosby} approximates by a combination of Brownian motion and sums of compound Poisson processes with two-sided exponential densities. In discrete time, Markov chains have been used to approximate the much larger class of Feller processes and \cite{bottcher} proves convergence in law of such an approximation in the Skorokhod space of c\`adl\`ag paths, but does not discuss rates of convergence; \cite{szimayer} has a finite state-space path approximation and applies this to option pricing together with a discussion of the rates of convergence for the prices. With respect to L\'evy process driven SDEs, \cite{higa} (resp. \cite{tanakahiga}) approximates solutions $Y$ thereto using a combination of a compound Poisson process and a high order scheme for the Brownian component (resp. discrete-time Markov chains and an operator approach) --- rates of convergence are then discussed for expectations of sufficiently regular real-valued functions against the marginals of the solutions. 

We remark that approximation/simulation of L\'evy processes in dimensions higher than one is in general more difficult than in the univariate case, see, e.g. the discussion on this in \cite{cohen} (which has a Gaussian approximation and establishes convergence in the Skorokhod space \cite[p. 197, Theorem 2.2]{cohen}). Observe also that in terms of pricing theory, the probability density function of a process can be viewed as the Arrow-Debreu state price, i.e.  the current value of an option whose payoff equals the Dirac delta function. The singular nature of this payoff makes it hard, particularly in the presence of jumps, to study the convergence of the prices under the discretised process to its continuous counterpart. 

Indeed, Theorem~\ref{theorem:convergence:multivariate} can be viewed as a generalisation of such convergence results for the well-known discretisation of the multi-dimensional Black-Scholes model (see e.g.~\cite{mijatovic} for the case of Brownian motion with drift in dimension one). In addition, existing literature, as specific to approximations of densities of L\'evy processes (or generalizations thereof), includes \cite{lopez} (polynomial expansion for a bounded variation driftless pure-jump process) and \cite{filipovic} (density expansions for multivariate affine jump-diffusion processes). \cite{knopova,sztonyk} study upper estimates for the densities.  On the other hand \cite{bally} has a result similar in spirit to ours, but for solutions to SDEs: for the case of the Euler approximation scheme, the authors there also study the rate of convergence of the transition densities. 

Further, from the point of view of partial integro-differential equations (PIDEs), the density $p:(0,\infty)\times \mathbb{R}^d\to [0,\infty)$ of the L\'evy process $X$ is the classical fundamental solution of the Cauchy problem (in $u\in C^{1,2}_0((0,\infty),\mathbb{R}^d)$) $\frac{\partial u}{\partial t}=Lu$, $L$ being the infinitesimal generator of $X$ \cite[Chapter 12]{conttankov} \cite[Chapter IV]{garronimenaldi}. Note that Assumption~\ref{assumption} in dimension one (resp. Assumption~\ref{assumption:multivariate} in the multivariate case) guarantees $p\in C^{1,\infty}_0$. There are numerous numerical methods in dealing with such PIDEs (and PIDEs in general): fast Fourier transform, trees and discrete-time Markov chains, viscosity solutions, Galerkin methods, see, e.g. \cite[Subsection 1.1]{contvoltchkova} \cite[Subsections 12.3-12.7]{conttankov} and the references therein. In particular, we mention the finite-difference method, which is in some sense the counterpart of the present article in the numerical analysis literature, discretising both in space and time, whereas we do so only in space. In general, this literature often restricts to finite activity processes, and either avoids a rigorous analysis of (the rates of) convergence, or, when it does, it does so for initial conditions $h=u(0,\cdot)$, which exclude the singular $\delta$-distribution. For example, \cite[p. 1616, Assumption 6.1]{contvoltchkova} requires $h$ continuous, piecewise $C^\infty$ with bounded derivatives of all orders; compare also Propositions~\ref{proposition:expectations} and~\ref{propositon:expectations:growth} concerning convergence of expectations in our setting. Moreover, unlike in our case where the discretisation is made outright, the approximation in \cite{contvoltchkova} is sequential, as is typical of the literature: beyond the restriction to a bounded domain (with boundary conditions), there is a truncation of the integral term in $L$, and then a reduction to the finite activity case, at which point our results are in agreement with what one would expect from the linear order of convergence of \cite[p. 1616, Theorem 6.7]{contvoltchkova}. 

\vspace{0.25cm}

\noindent The rest of the paper is organised as follows. Section~\ref{section:Definitions} introduces the setting by specifying the Markov generator of $X^h$ and precisely states the main results. Then Section~\ref{section:Kernels} provides integral expressions for the transition kernels by applying spectral theory to the generator of the approximating chain and studies the convergence of the characteristic exponents. In section~\ref{section:Rates} this allows us to establish convergence rates for the transition densities. While Sections~\ref{section:Kernels} and~\ref{section:Rates} restrict this analysis to the univariate case, explicit comments are made in both, on how to extend the results to the multivariate setting (this extension being, for the most part, direct and trivial). Finally, Section~\ref{section:expectations:numerics} derives some results regarding convergence of expectations $\EE[f\circ X^h_t]\to \EE[f\circ X_t]$ for suitable test functions $f$, presents a numerical algorithm, under which computations are eventually done, discusses the corresponding truncation/localization error and gives some numerical experiments.

\section{Definitions, notation and statement of results}\label{section:Definitions}

\subsection{Setting}

Fix a dimension $d\in \mathbb{N}$ and let $(e_j)_{j=1}^d$ be the standard orthonormal basis of $\mathbb{R}^d$. Further, let $X$ be an $\mathbb{R}^d$-valued L\'evy process with characteristic exponent \cite[pp. 37-39]{sato}: 
\begin{equation}\label{eq:characteristic_exponent}
\Psi(p)=-\frac{1}{2}\left\langle p,\Diffusion p\right\rangle+i\langle \Drift, p\rangle+\int_{\mathbb{R}^d}\left(e^{i\langle p,x\rangle}-i\langle p,x\rangle\Ind(x)-1\right)d\Measure(x)
\end{equation}
($p\in\mathbb{R}^d$). Here $(\Diffusion,\Measure,\Drift)_{\tilde{c}}$ is the characteristic triplet relative to the cut-off function $\tilde{c}=\Ind$; $V$ is 1 or 0, the latter only if $\int_{[-1,1]^d}\vert x\vert d\lambda(x)<\infty$. Note that $X$ is then a Markov process with transition function $P_{t,T}(x,B):=\PP(X_{T-t}\in B-x)$ ($0\leq t\leq T$, $x\in\mathbb{R}^d$ and $B\in\mathcal{B}(\mathbb{R}^d)$) and (for $t\geq 0$, $p\in \mathbb{R}^d$) $\phi_{X_t}(p):=\EE[e^{ip X_t}]=\exp\{t\Psi(p)\}$. We refer to~\cite{bertoin,sato} for the general background on L\'evy processes.

Since $\Diffusion\in \mathbb{R}^d$ is symmetric, nonnegative definite, it is assumed without loss of generality that $\Diffusion=\mathrm{diag}(\diffusion_1,\ldots,\diffusion_d)$ with $ \diffusion_1\geq \cdots\geq \diffusion_d$. We let $l:=\max\{k\in \{1,\ldots,d\}:\diffusion_k>0\}$ ($\max\emptyset:=0$). In the univariate case $d=1$, $\Diffusion$ reduces to the scalar $\diffusion:=\diffusion_1$. 

Now fix $h>0$. Consider a CTMC $X^h=(X^h_t)_{t\geq 0}$ approximating our L\'evy process $X$ (in law). We describe $X^h$ as having \cite{norris} state space $\Zh^d:=h\mathbb{Z}^d:=\{hk:k\in\mathbb{Z}^d\}$ ($\Zh:=\Zh^1$), initial state $X_0^h=0$, a.s. and an infinitesimal generator $L^h$ given by a spatially homogeneous $Q$-matrix $Q^h$ (i.e. $Q^h_{ss'}$ depends only on $s-s'$, for $\{s,s'\}\subset \Zh^d$). Thus $L^h$ is a mapping defined on the set $l^\infty(\Zh^d)$ of bounded functions $f$ on $\Zh^d$, and $L^hf(s)=\sum_{s'\in S}Q^h_{ss'}f(s')$. 

It remains to specify $Q^h$. To this end we discretise on $\Zh^d$ the infinitesimal generator $L$ of the L\'evy process $X$, thus obtaining $L^h$. Recall that \cite[p. 208, Theorem 31.5]{sato}: $$Lf(x)=\sum_{j=1}^d\left(\frac{\diffusion_j}{2}\partial_{jj}f(x)+\mu_j\partial_jf(x)\right)+\int_{\mathbb{R}^d}\left(f(x+y)-f(x)-\sum_{j=1}^dy_j\partial_jf(x)\Ind(y)\right)d\Measure(y)$$ ($f\in C^2_0(\mathbb{R}^d)$, $x\in \mathbb{R}^d$). We specify $L^h$ separately in the univariate, $d=1$, and in the general, multivariate, setting.

\subsubsection{Univariate case}\label{susubsection:setting:univariate}
In the case when $d=1$, we introduce two schemes. Referred to as \textbf{discretization scheme 1} (resp. \textbf{2}), and given by \eqref{equation:infinitesimalgeneratorforY} (resp. \eqref{equation:infinitesimalgeneratorforY2}) below, they differ in the discretization of the first derivative, as follows. 

Under \textbf{discretisation scheme 1}, for $s\in \Zh$ and $f:\Zh\to\mathbb{R}$ vanishing at infinity:
\begin{eqnarray}
\nonumber
L^hf(s) & = &\frac{1}{2}\left(\sigma^2+c^h_0\right)\frac{f(s+h)+f(s-h)-2f(s)}{h^2}+\left(\mu-\mu^h\right)\frac{f(s+h)-f(s-h)}{2h}\\
& + & \sum_{s'\in \Zh\backslash \{0\}}\left[f(s+s')-f(s)\right]c^h_{s'} 
\label{equation:infinitesimalgeneratorforY}
\end{eqnarray}
where the following notation has been introduced: 
\begin{itemize}
\item for $s\in \Zh$: 
$$A_s^h := \begin{cases} 
      [s-h/2,s+h/2), & \textrm{ if $s<0$} \\
      [-h/2,h/2], & \textrm{ if $s=0$} \\
      (s-h/2,s+h/2], & \textrm{ if $s>0$} \\
   \end{cases};
$$
\item for $s\in \Zh\backslash\{0\}$: $c^h_{s}:=\lambda(A_s^h)$;
\item and finally: $$c^h_0:=\int_{A_0^h}y^2\ind(y)d\lambda(y)\hspace{0.5cm}\text{and}\hspace{0.5cm} \mu^h:=\sum_{s\in \Zh\backslash \{0\}}s\int_{A_s^h}\ind(y)d\lambda(y).$$
\end{itemize}

Note that $Q^h$ has nonnegative off-diagonal entries for all $h$ for which: 
\begin{equation}\label{eq:conditionsongenerator}
\frac{\sigma^2+c^h_0}{2h^2}+\frac{\mu-\mu^h}{2h}+c^h_h\geq 0\qquad \mathrm{and}\qquad \frac{\sigma^2+c^h_0}{2h^2}-\frac{\mu-\mu^h}{2h}+c^h_{-h}\geq 0
\end{equation}
and in that case $Q^h$ is a genuine $Q$-matrix. Moreover, due to spatial homogeneity, its entries are then also uniformly bounded in absolute value. 

Further, when $\sigma^2>0$, it will be shown that (\ref{eq:conditionsongenerator}) always holds, at least for all sufficiently small $h$ (see Proposition~\ref{proposition:conditionQmatrix}). However, in general, \eqref{eq:conditionsongenerator} may fail. It is for this reason that we introduce scheme 2, under which the condition on the nonnegativity of off-diagonal entries of $Q^h$ holds vacuously.

To wit, we use in \textbf{discretization scheme 2} the one-sided, rather than the two-sided discretisation of the first derivative, so that (\ref{equation:infinitesimalgeneratorforY}) reads:
\begin{eqnarray}\label{equation:infinitesimalgeneratorforY2}
\nonumber L^hf(s)&=&\frac{1}{2}\left(\sigma^2+c^h_0\right)\frac{f(s+h)+f(s-h)-2f(s)}{h^2}+\sum_{s'\in \Zh\backslash \{0\}}[f(s+s')-f(s)]c^h_{s'}+\\
&+&\!\!\!(\mu-\mu^h)\left(\frac{f(s+h)-f(s)}{h}\mathbbm{1}_{[0,\infty)}(\mu-\mu^h)+\frac{f(s)-f(s-h)}{h}\mathbbm{1}_{(-\infty,0]}(\mu-\mu^h)\right)
\end{eqnarray}

Importantly, while scheme 2 is always well-defined, scheme 1 is not; and yet the two-sided discretization of the first derivative exhibits better convergence properties than the one-sided one (cf. Proposition~\ref{proposition:someestimates}). We therefore retain the treatment of both these schemes in the sequel. 

For ease of reference we also summarize here the following notation which will be used from Subsection~\ref{exponentials} onwards: $$c:=\lambda(\mathbb{R}),\quad b := \kappa(0),\quad d:=\lambda(\mathbb{R}\backslash [-1,1])$$ and for $\delta\in (0, 1]$: $$\zeta(\delta) := \delta\int_{[-1,1]\backslash [-\delta,\delta]} |x|d\lambda(x)\text{ and } \gamma(\delta) := \delta^2\int_{[-1,1]\backslash [-\delta,\delta]} d\lambda(x).$$

\subsubsection{Multivariate case}\label{sububsection:setting:multivariate}
For the sake of simplicity we introduce only one discretisation scheme in this general setting. If necessary, and to avoid confusion, we shall refer to it as the \textbf{multivariate scheme}. We choose $V=0$ or $V=1$, according as $\Measure(\mathbb{R}^d)$ is finite or infinite. $L^h$ is then given by: 

\begin{eqnarray*}
L^hf(s)&=&\frac{1}{2}\sum_{j=1}^d\left(\diffusion_j+c_{0j}^h\right)\frac{f(s+he_j)+f(s-he_j)-2f(s)}{h^2}+\sum_{j=1}^l(\Drift_j-\Drift^h_j)\frac{f(s+he_j)-f(s-he_j)}{2h}+\\
&&\sum_{j=l+1}^d(\Drift_j-\Drift_j^h)\left(\frac{f(s+he_j)-f(s)}{h}\mathbbm{1}_{[0,\infty)}(\Drift_j-\Drift_j^h)+\frac{f(s)-f(s-he_j)}{h}\mathbbm{1}_{(-\infty,0]}(\Drift_j-\Drift_j^h)\right)+\\
&&\sum_{s'\in\Zh^d}\left(f(s+s')-f(s)\right)c_{s'}^h
\end{eqnarray*}
($f\in c_0(\Zh^d)$, $s\in\Zh^d$; and we agree $\sum_\emptyset:=0$). Here the following notation has been introduced: 
\begin{itemize}
\item for $s\in \Zh^d$: $A_s^h:=\prod_{j=1}^d I_{s_j}^h$, where for $s\in \Zh$:
$$I_s^h := \begin{cases} 
      [s-h/2,s+h/2), & \textrm{ if $s<0$} \\
      [-h/2,h/2], & \textrm{ if $s=0$} \\
      (s-h/2,s+h/2], & \textrm{ if $s>0$} \\
   \end{cases}
$$
so that $\{A_s^h:s\in \Zh^d\}$ constitutes a partition of $\mathbb{R}^d$;
\item for $s\in \Zh^d\backslash\{0\}$: $c^h_{s}:=\lambda(A_s^h)$;
\item and finally for $j\in \{1,\ldots,d\}$: $$c_{0j}^h:=\int_{A_0^h}x_j^2\Ind(x)d\Measure(x)\qquad \text{and}\qquad\mu^h_j:=\sum_{s\in \Zh^d\backslash \{0\}}s_j\int_{A_s^h}\Ind(y)d\lambda(y).$$
\end{itemize}
Notice that when $d=1$, this scheme reduces to scheme 1 or scheme 2, according as $\diffusion>0$ or $\diffusion=0$. Indeed, statements pertaining to the multivariate scheme will always be understood to include also the univariate case $d=1$.

\begin{remark}\label{remark:multivariate:setting}
The complete analogue of $c^h_0$ from the univariate case would be the matrix $c_0^h$, entries $(c_0^h)_{ij}:=\int_{A_0^h}x_ix_j\Ind(x)d\Measure(x)$, $\{i,j\}\subset \{1,\ldots,d\}$. However, as $h$ varies, so could $c_0^h$, and thus no diagonalization of $c_0^h+\Diffusion$ possible (in general), simultaneously in all (small enough) positive $h$. Thus, retaining $c_0^h$ in its totality, we should have to discretize mixed second partial derivatives, which would introduce (further) nonpositive entries in the corresponding $Q$-matrix $Q^h$ of $X^h$. It is not clear whether these would necessarily be counter-balanced in a way that would ensure nonnegative off-diagonal entries. Retaining the diagonal terms of $c_0^h$, however, is of no material consequence in this respect.
\end{remark}
It is verified just as in the univariate case, component by component, that there is some $h_\star\in (0,+\infty]$ such that for all $h\in (0,h_\star)$, $L^h$ is indeed the infinitesimal generator of some CTMC (i.e. the off-diagonal entries of $Q^h$ are nonnegative). $Q^h$ is then a regular (as spatially homogeneous) $Q$-matrix, and $X^h$ is a compound Poisson process, whose L\'evy measure we denote $\Measure^h$.

\subsection{Summary of results}
We have, of course: 

\begin{remark}[Convergence in distribution]\label{theorem:convergenceindistribution}
$X^h$ converges to $X$ weakly in finite-dimensional distributions (hence w.r.t. the Skorokhod topology on the space of c\`adl\`ag paths \cite[p. 415, 3.9 Corollary]{jacod}) as $h\downarrow 0$.
\end{remark}

Next, in order to formulate the rates of convergence, recall that $P^h_{t,T}(x,y)$ (resp. $p_{t,T}(x,y)$) denote the transition probabilities (resp. continuous transition densities, when they exist) of $X^h$ (resp. $X$) from $x$ at time $t$ to $y$ at time $T$, $\{x,y\}\subset \Zh^d$, $0\leq t<T$. Further, for $0\leq t<T$ define: 
\begin{equation}\label{eq:fundamental_notation}
\Delta_{T-t}(h):=\sup_{\{x,y\}\subset \Zh^d}D^h_{t,T}(x,y)\qquad\text{where}\qquad D^h_{t,T}(x,y):=\left\vert p_{t,T}(x,y)-\frac{1}{h^d}P^h_{t,T}(x,y)\right\vert.
\end{equation} 

We now summarize the results first in the univariate, and then in the multivariate setting (Remark~\ref{theorem:convergenceindistribution} holding true of both). 

\subsubsection{Univariate case}\label{subsubsection:results:univariate}
The assumption alluded to in the introduction is the following (we state it explicitly when it is being used):

\begin{assumption}\label{assumption}
Either $\sigma^2>0$ or Orey's condition holds:
$$\exists \epsilon\in (0,2)\qquad\text{such that}\qquad \liminf_{r\downarrow 0}\frac{1}{r^{2-\epsilon}}\int_{[-r,r]}u^2d\lambda(u)>0.$$
\end{assumption}
The usage of the two schemes and the specification of $V$ is as summarized in Table~\ref{table:reference}. In short we use scheme 1 or scheme 2, according as $\sigma^2>0$ or $\sigma^2=0$, and we use $V=0$ or $V=1$, according as $\lambda(\mathbb{R})<\infty$ or $\measure(\mathbb{R})=\infty$. By contrast to Assumption~\ref{assumption} we maintain Table~\ref{table:reference} as being in effect throughout this subsubsection.

\begin{table}[!hbt]
\caption{Usage of the two schemes and of $V$ depending on the nature of $\sigma^2$ and $\lambda$.}\label{table:reference}
\begin{center}
\begin{tabular}{|c|c|c|}\hline
 L\'evy measure/diffusion part& $\sigma^2>0$  & $\sigma^2=0$ \\\hline
$\lambda(\mathbb{R})<\infty$ &  scheme 1, $V=0$ & scheme 2, $V=0$\\\hline
$\lambda(\mathbb{R})=\infty$  & scheme 1, $V=1$ & scheme 2, $V=1$\\\hline
\end{tabular}
\end{center}
\end{table}
Under Assumption~\ref{assumption} for every $t>0$, $\phi_{X_t}\in L^1(m)$ where $m$ is Lebesgue measure and (for $0\leq t<T$, $\{x,y\}\subset\mathbb{R}$):

\begin{equation}\label{equation:trnasitionkernelcont}
p_{t,T}(x,y)=\frac{1}{2\pi}\int_\mathbb{R}\exp\left\{ ip(x-y)\right\}\exp\left\{\Psi(p)(T-t)\right\}dp
\end{equation} 
(cf. Remark~\ref{proposition:levydensity}). Similarly, with $\Psi^h$ denoting the characteristic exponent of the compound Poisson process $X^h$ (for $0\leq t<T$, $y\in \Zh$, $ \PP_{X_t^h}$-a.s. in $x\in \Zh$):
 
\begin{equation}\label{equation:trnasitionkerneldiscrete}
\frac{1}{h}P_{t,T}^h(x,y)=\frac{1}{2\pi}\int_{-\frac{\pi}{h}}^{\frac{\pi}{h}}\exp\{ip(x-y)\}\exp\{\Psi^h(p)(T-t)\}dp.
\end{equation}
Note that the right-hand side is defined even if $\PP(X^h_t = x)=0$ and we let the left-hand side take this value when this is so. 

The main result can now be stated.

\begin{theorem}[Convergence of transition kernels]\label{theorem:convergencerates}
Under Assumption~\ref{assumption}, whenever $s>0$, the convergence of $\Delta_s(h)$ is as summarized in the following table. In general convergence is no better than stipulated. 
\begin{center}
\begin{tabular}{|c|c|c|c|c|}\hline
 & $\lambda(\mathbb{R})=0$ & $0<\lambda(\mathbb{R})<\infty$ & $\kappa(0)<\infty=\lambda(\mathbb{R})$ & $\kappa(0)= \infty$\\\hline
$\sigma^2>0$ & $\Delta_s(h)=O(h^2)$ & $\Delta_s(h)=O(h)$ & \multirow{2}{*}{$\Delta_s(h)=O(h)$} & \multirow{2}{*}{$\Delta_s(h)=O(h\kappa(h/2))$}\\\cline{1-3}
$\sigma^2=0$ & $\times$ & $\times$ & &\\\hline
\end{tabular}
\end{center}
\end{theorem}
More exhaustive statements, of which this theorem is a summary, are to be found in Propositions~\ref{propositon:convergence_kernels_diffusion_positive} and~\ref{propositon:convergence_kernels_diffusion_zero}, and will be proved in Section~\ref{section:Rates}. 

\subsubsection{Multivariate case}\label{subsubsection:results:multivariate}
The relevant technical condition here is: 

\begin{assumption}\label{assumption:multivariate}
There are $\{P,C,\epsilon\}\subset (0,\infty)$ and an $h_0\in (0,h_\star]$, such that for all $h\in (0,h_0)$, $s>0$ and $p\in [-\pi/h,\pi/h]^d\backslash (-P,P)^d$: 
\begin{equation}\label{eq:assumption:multivariate:one}
\vert \phi_{X_s^h}(p)\vert\leq \exp\{-Cs\vert p\vert^\epsilon\}
\end{equation}
whereas for $p\in \mathbb{R}^d\backslash (-P,P)^d$: 
\begin{equation}\label{eq:assumption:multivariate:two}
\vert \phi_{X_s}(p)\vert\leq \exp\{-Cs\vert p\vert^\epsilon\}.
\end{equation}
\end{assumption}
Again we shall state it explicitly when it is being used. 

\begin{remark}\label{remark:multiariate:condition_discussion}
It is shown, just as in the univariate case, that Assumption~\ref{assumption:multivariate} holds if $l=d$, i.e. if $\Diffusion$ is non-degenerate. Moreover, then we may take $P=0$, $C=\frac{1}{2}\left(\frac{2}{\pi}\right)^2\left(\land_{j=1}^d\diffusion_j\right)$, $\epsilon=2$ and $h_0=h_\star$. 

It would be natural to expect that the same could be verified for the multivariate analogue of Orey's condition, which we suggest as being: $$\liminf_{\delta\downarrow 0}\inf_{e\in S^{d-1}} \int_{\overline{B}(0,r)}\vert \langle e,x\rangle\vert^2d\Measure(x)/r^{2-\epsilon}>0$$ for some $\epsilon\in (0,2)$ (with $S^{d-1}\subset\mathbb{R}^d$ (resp. $\overline{B}(0,r)\subset \mathbb{R}^d$) the unit sphere (resp. closed ball of radius $r$ centered at the origin)). Specifically, it is easy to see that \eqref{eq:assumption:multivariate:two} of Assumption~\ref{assumption:multivariate} still holds. However, we are unable to show the validity of \eqref{eq:assumption:multivariate:one}.
\end{remark}
Under Assumption~\ref{assumption:multivariate}, Fourier inversion yields the integral representation of the continuous transition densities for $X$ (for $0\leq t<T$, $\{x,y\}\subset \mathbb{R}^d$): 
\begin{equation*}
p_{t,T}(x,y)=\frac{1}{(2\pi)^d}\int_{\mathbb{R}^d}e^{i\langle p,x-y\rangle}e^{(T-t)\Psi(p)}dp.
\end{equation*}
On the other hand, $L^2([-\pi/h,\pi/h]^d)$ Hilbert space techniques yield for the normalized transition probabilities of $X^h$ (for $0\leq t<T$, $y\in \Zh^d$ and $\PP_{X_t^h}$-a.s. in $x\in\Zh^d$): 
\begin{equation*}
\frac{1}{h^d}P_{t,T}(x,y)=\frac{1}{(2\pi)^d}\int_{[-\pi/h,\pi/h]^d}e^{i\langle p,x-y\rangle}e^{(T-t)\Psi^h(p)}dp,
\end{equation*}
where $\Psi^h$ is the characteristic exponent of $X^h$. 

Finally, we state the result with the help of the following notation:
\begin{itemize}
\item for $\delta\in [0,\infty)$: $\kappa(\delta):=\int_{[-1,1]^d\backslash [-\delta,\delta]^d}\vert x\vert d\Measure(x)$, $\zeta(\delta):=\delta\kappa(\delta)\text{ and }\chi(\delta):=\sum_{1\leq i<j\leq d}\int_{[-\delta,\delta]^d}\vert x_ix_j\vert d\Measure(x)$. 
\item $\hatdiffusion:=\land_{j=1}^d\diffusion_j\text{ and }\diffusion:=\sum_{j=1}^d\diffusion_j$.
\end{itemize}
Note that by the dominated convergence theorem, $(\zeta+\chi)(\delta)\to 0$ as $\delta\downarrow 0$ (this is seen as in the univariate case, cf. Lemma~\ref{lemma:somethingsvnaish}).

\begin{theorem}[Convergence --- multivariate case]\label{theorem:convergence:multivariate}
Let $d\in \mathbb{N}$ and suppose Assumption~\ref{assumption:multivariate} holds. Then for any $s>0$, $\Delta_s(h)=O(h\lor (\zeta+\chi)(h/2))$. Moreover, if $\hatdiffusion>0$, then there exists a universal constant $D_d\in (0,\infty)$, such that:
\begin{enumerate}
\item if $\Measure(\mathbb{R}^d)=0$, $$\limsup_{h\downarrow 0}\frac{\Delta_s(h)}{h^2}\leq D_d\left[\frac{\diffusion}{\hatdiffusion}\frac{1}{\sqrt{s\hatdiffusion}}+\frac{\vert \mu\vert}{\hatdiffusion}\right]\frac{1}{(s\hatdiffusion)^{\frac{d+1}{2}}}.$$
\item if $0<\Measure(\mathbb{R}^d)<\infty$, $$\limsup_{h\downarrow 0}\frac{\Delta_s(h)}{h}\leq D_d\frac{\Measure(\mathbb{R}^d)s}{(s\hatdiffusion)^{\frac{d+1}{2}}}.$$
\item if $\kappa(0)<\infty=\Measure(\mathbb{R}^d)$,  $$\limsup_{h\downarrow 0}\frac{\Delta_s(h)}{h}\leq D_d\left[\Measure(\mathbb{R}^d\backslash [-1,1]^d)s+\frac{\kappa(0)s}{\sqrt{s\hatdiffusion}}\right]\frac{1}{(s\hatdiffusion)^{\frac{d+1}{2}}}.$$
\item if $\kappa(0)=\infty$,   $$\limsup_{h\downarrow 0}\frac{\Delta_s(h)}{(\zeta+\chi)(h/2)}\leq D_d\frac{s}{(s\hatdiffusion)^{\frac{d+2}{2}}}.$$
\end{enumerate}
\end{theorem}
\begin{remark}
Notice that in the univariate case $\zeta+\chi$ reduces to $\zeta$. The presence of $\chi$ is a consequence of the omission of non-diagonal entries of $c_0^h$ in the multivariate approximation scheme (cf. Remark~\ref{remark:multivariate:setting}). 
\end{remark}
The proof of Theorem~\ref{theorem:convergence:multivariate} is an easy extension of the arguments behind Theorem~\ref{theorem:convergencerates}, and we comment on this immediately following the proof of Proposition~\ref{proposition:transfer}.

\section{Transition kernels and convergence of characteristic exponents}\label{section:Kernels}
In the interest of space, simplicity of notation and ease of exposition, the analysis in this and in Section~\ref{section:Rates} is \emph{restricted} to dimension $d=1$. Proofs in the multivariate setting are, for the most part, a direct and trivial extension of those in the univariate case. However, when this is not so, necessary and explicit comments will be provided in the sequel, as appropriate. 

\subsection{Integral representations}\label{subsection:integralrep}
First we note the following result (its proof is essentially by the standard inversion theorem, see also \cite[p. 190, Proposition 28.3]{sato}). 

\begin{remark}\label{proposition:levydensity}
Under Assumption~\ref{assumption}, for some $\{P,C,\epsilon\}\subset (0,\infty)$ depending only on $\{\lambda,\sigma^2\}$  and then all $p\in \mathbb{R}\backslash (-P,P)$ and $t\geq 0$: $\vert \phi_{X_t}(p)\vert\leq \exp\{-Ct\vert p\vert^\epsilon\}$. Moreover, when $\sigma^2>0$, one may take $P=0$, $C=\frac{1}{2}\diffusion$ and $\epsilon=2$, whereas otherwise $\epsilon$ may take the value from Orey's condition in Assumption~\ref{assumption}. Consequently, $X_t$ ($t>0$) admits the continuous density $f_{X_t}(y)=\frac{1}{2\pi}\int_\mathbb{R}e^{-ipy}\phi_{X_t}(p)dp$ ($y\in\mathbb{R}$). In particular, the law $P_{t,T}(x,\cdot)$ is given by (\ref{equation:trnasitionkernelcont}).  
\end{remark} 
Second, to obtain (\ref{equation:trnasitionkerneldiscrete}) we apply some classical theory of Hilbert spaces, see e.g.~\cite{dudley}.

\begin{definition}\label{definition:semidiscrete}
For $s\in \Zh$ let $g_s:[-\frac{\pi}{h},\frac{\pi}{h}]\to\mathbb{C}$ be given by $g_s(p):=\sqrt{\frac{h}{2\pi}}e^{-isp}$. The $(g_s)_{s\in \Zh}$ constitute an orthonormal basis of the Hilbert space $L^2([-\frac{\pi}{h},\frac{\pi}{h}])$. 

Let $A\in l^2(\Zh)$. We define:
$\mathcal{F}_hA:=\sum_{s\in \Zh}A(s)g_s$. The inverse of this transform $\mathcal{F}_h^{-1}:L^2([-\frac{\pi}{h},\frac{\pi}{h}])\to l^2(\Zh)$ is given by:
$$(\mathcal{F}^{-1}_h\phi)(s)=\langle \phi,g_s\rangle:=\int_{[-\frac{\pi}{h},\frac{\pi}{h}]}\phi\overline{g_s}dm$$  for $\phi\in L^2([-\frac{\pi}{h},\frac{\pi}{h}])$ and $s\in \Zh$.
\end{definition}

\begin{definition}
For a bounded linear operator $A:l^2(\Zh)\to l^2(\Zh)$, we say $F_A:[-\pi/h,\pi/h]\to\mathbb{R}$ is its \emph{diagonalization}, if $\mathcal{F}_hA\mathcal{F}_h^{-1}\phi=F_{A}\phi$ for all $\phi\in L^2([-\frac{\pi}{h},\frac{\pi}{h}])$.
\end{definition}
We now diagonalize $L^h$, which allows us to establish (\ref{equation:trnasitionkerneldiscrete}). The straightforward proof is left to the reader.

\begin{proposition}\label{proposition:somespectralrrepresentations}
Fix $C\in l^1(\Zh)$. The following introduces a number of bounded linear operators $A:l^2(\Zh)\to l^2(\Zh)$ and gives their diagonalization. With $f\in l^2(\Zh)$, $s\in \Zh$, $p\in [-\frac{\pi}{h},\frac{\pi}{h}]$:  
\begin{enumerate}[(i)]
\item $\Delta_hf(s):=\frac{f(s+h)+f(s-h)-2f(s)}{h^2}$. $F_{\Delta_h}(p)=2\frac{\cos(hp)-1}{h^2}$.
\item $\nabla_hf(s):=\frac{f(s+h)-f(s-h)}{2h}$. $F_{\nabla_h}(p)=i\frac{\sin(hp)}{h}$. 
Under scheme 2 we let $\nabla_h^+f(s):=\frac{f(s+h)-f(s)}{h}$ (resp. $\nabla_h^-f(s):=\frac{f(s)-f(s-h)}{h}$)
and then $F_{\nabla_h^+}(p)=\frac{e^{ihp}-1}{h}$ (resp. $F_{\nabla_h^-}(p)=\frac{1-e^{-ihp}}{h}$).
\item $L_Cf(s):=\sum_{s'\in \Zh}(f(s+s')-f(s))C(s')$. $F_{L_C}(p)=\sum_{s\in \Zh} C(s)(e^{isp}-1)$.
\end{enumerate}
\end{proposition}

As $\lambda$ is finite outside any neighborhood of $0$, $L^h\vert_{l^2(\Zh)}$ (as in (\ref{equation:infinitesimalgeneratorforY}), resp. (\ref{equation:infinitesimalgeneratorforY2})) is a bounded linear mapping. We denote this restriction by $L^h$ also. Its diagonalization is then given by $\Psi^h:=F_{L^h}$, where, under scheme 1,

\begin{equation}\label{eq:spectralthingy2}
\Psi^h(p)=i(\mu-\mu^h) \frac{\sin(hp)}{h}+(\sigma^2+c_0^h)\frac{(\cos(hp)-1)}{h^2}+\sum_{s\in \Zh\backslash \{0\}}c_s^h\left(e^{isp}-1\right)
\end{equation}
and under scheme 2,
\begin{eqnarray} \nonumber
\Psi^h(p)&=&(\mu-\mu^h)\left(\frac{e^{ihp}-1}{h}\mathbbm{1}_{[0,\infty)}(\mu-\mu^h)+\frac{1-e^{-ihp}}{h}\mathbbm{1}_{(-\infty,0]}(\mu-\mu^h)\right)+\\
& + &(\sigma^2+c_0^h)\frac{(\cos(hp)-1)}{h^2}+\sum_{s\in \Zh\backslash \{0\}}c_s^h\left(e^{isp}-1\right)
\label{eq:spectralthingy}
\end{eqnarray}
(with $p\in [-\frac{\pi}{h},\frac{\pi}{h}]$, but we can and will view $\Psi^h$ as defined for all real $p$ by the formulae above). Under either scheme, $\Psi^h$ is bounded and continuous as the final sum converges absolutely and uniformly. 

\begin{proposition}\label{propositoin:transition_probablities}
For scheme 1 under (\ref{eq:conditionsongenerator}) and always for scheme 2, for every $0\leq t<T$, $y\in \Zh$ and $\PP_{X_t^h}$-a.s. in $x\in \Zh$ (\ref{equation:trnasitionkerneldiscrete}) holds, i.e.:
$$\PP(X_T^h=y\vert X_t^h=x)=\frac{h}{2\pi}\int_{-\frac{\pi}{h}}^{\frac{\pi}{h}}\exp\{ip(x-y)\}\exp\{\Psi^h(p)(T-t)\}dp.$$
\end{proposition}
\begin{proof} (Condition \eqref{eq:conditionsongenerator} ensures scheme 1 is well-defined ($Q^h$ needs to have nonnegative off-diagonal entries).) Note that: $\PP(X_T^h=y\vert X_t^h=x)=(e^{(T-t)L^h}\mathbbm{1}_{\{y\}})(x)$. Thus \eqref{equation:trnasitionkerneldiscrete} follows directly from the relation $\mathcal{F}_hL^h\mathcal{F}_h^{-1}=\Psi^h\cdot$ (where $\Psi^h\cdot$ is the operator that multiplies functions pointwise by $\Psi^h$).
\end{proof}
In what follows we study the convergence of~(\ref{equation:trnasitionkernelcont}) to~(\ref{equation:trnasitionkerneldiscrete}) as $h\downarrow 0$. These expressions are particularly suited to such an analysis, not least of all because the spatial and temporal components are factorized. 

One also checks that for every $t\geq 0$ and $p\in\mathbb{R}$: $$\phi_{X_t^h}(p)=\EE[e^{ipX_t^h}]=\exp\{t\Psi^h(p)\}.$$ Hence $X^h$ are compound Poisson processes \cite[p. 18, Definition 4.2]{sato}.

In the multivariate scheme, by considering the Hilbert space $L^2([-\pi/h,\pi/h]^d)$ instead, $X^h$ is again seen to be compound Poisson with characteristic exponent given by (for $p\in \mathbb{R}^d$): 
\begin{eqnarray}
\nonumber \Psi^h(p)&=&\sum_{j=1}^d(\diffusion_j+c_{0j}^h)\frac{\cos(hp_j)-1}{h^2}+i \sum_{j=1}^l(\Drift_j-\Drift_j^h)\frac{\sin(hp_j)}{h}\\
\nonumber &+& \sum_{j=l+1}^d(\Drift_j-\Drift_j^h)\left(\frac{e^{ihp_j}-1}{h}\mathbbm{1}_{[0,\infty)}(\Drift_j-\Drift_j^h)+\frac{1-e^{-ihp_j}}{h}\mathbbm{1}_{(-\infty,0]}(\Drift_j-\Drift_j^h)\right)\\
&+&\sum_{s\in \Zh^d\backslash \{0\}}\left(e^{i\langle p,s\rangle}-1\right)c_s^h.
\label{eq:multivariate:exponent}
\end{eqnarray}

In the sequel, we shall let $\lambda^h$ denote the L\'evy measure of $X^h$. 

\subsection{Convergence of characteristic exponents}\label{exponentials}
We introduce for $p\in\mathbb{R}$:
\begin{equation*}
 f_h(p) := \frac{\cos(hp)-1}{h^2}+\frac{p^2}{2}
\end{equation*}
and, under scheme 1: 
\begin{eqnarray*}
g_h(p)& := & i\left(\frac{\sin(hp)}{h}-p\right)\\
l_h(p) & := & c_0^h\frac{\cos(hp)-1}{h^2}-\mu^hi\frac{\sin(hp)}{h}+\!\!\!
\sum_{s\in \Zh\backslash\{0\}}c_s^h\left(e^{isp}-1\right)-\!\!\int_{\mathbb{R}}\left(e^{ipu}-1-ipu\ind(u)\right)d\lambda(u),
\end{eqnarray*}
respectively, under scheme 2:
\begin{eqnarray*}
g_h(p)& := &\frac{e^{ihp}-1}{h}\mathbbm{1}_{(0,\infty)}(\mu-\mu^h)+\frac{1-e^{-ihp}}{h}\mathbbm{1}_{(-\infty,0]}(\mu-\mu^h)-ip;\\
l_h(p) & := & c_0^h\frac{\cos(hp)-1}{h^2}-\mu^h \left[\frac{e^{ihp}-1}{h}\mathbbm{1}_{(0,\infty)}(\mu-\mu^h)+\frac{1-e^{-ihp}}{h}\mathbbm{1}_{(-\infty,0]}(\mu-\mu^h)\right]+\\
&+&\sum_{s\in \Zh\backslash\{0\}}c_s^h\left(e^{isp}-1\right)-\int_{\mathbb{R}}\left(e^{ipu}-1-ipu\ind(u)\right)d\lambda(u).
\end{eqnarray*}
Thus: $$\Psi^h-\Psi=\sigma^2f_h+\mu g_h+l_h.$$

Next, three elementary but key lemmas. The first concerns some elementary trigonometric inequalities as well as the Lipschitz difference for the remainder of the exponential series $f_l(x):=\sum_{k=l+1}^\infty\frac{(ix)^k}{k!}$ ($x\in \mathbb{R}$, $l\in \{0,1,2\}$): these estimates will be used again and again in what follows. The second is only used in the estimates pertaining to the multivariate scheme. Finally, the third lemma establishes key convergence properties relating to $\lambda$.

\begin{lemma}\label{lemma:someestimates} 
For all real $x$: $0\leq \cos(x)-1+\frac{x^2}{2}\leq \frac{x^4}{4!}$, $0\leq \sgn(x)(x-\sin(x))\leq \sgn(x)\frac{x^3} {3!}$ and $0\leq x^2+2(1-\cos(x))-2x\sin(x)\leq x^4/4.$ Whenever $\{x,y\}\subset \mathbb{R}$ we have (with $\delta:=y-x$):
\begin{enumerate} 
\item\label{lemma:someestimates:i} $\vert e^{ix}-1-(e^{iy}-1)\vert^2\leq \delta^2$. 
\item\label{lemma:someestimates:ii} $\vert e^{ix}-1-ix-(e^{iy}-1-iy)\vert^2\leq \delta^4/4+\delta^2x^2+\vert \delta\vert^3\vert x\vert$.
\item\label{lemma:someestimates:iii} $\vert e^{ix}-1-ix+x^2/2-(e^{iy}-1-iy+y^2/2)\vert^2\leq \delta^6/36+\vert \delta\vert^5\vert x\vert/6+(5/12)\delta^4x^2+\vert\delta\vert^3\vert x\vert^3/2+\delta^2x^4/4$.
\end{enumerate}
\end{lemma}
\begin{proof}
The first set of inequalities may be proved by comparison of derivatives. Then, \eqref{lemma:someestimates:i} follows from $\vert e^{i(x-y)}-1\vert^2=2(1-\cos(x-y))$ and $\vert e^{iy}\vert=1$; \eqref{lemma:someestimates:ii} from $$\vert e^{ix}-ix-e^{iy}+iy\vert^2=\left(\delta^2+2(1-\cos(\delta))-2\delta\sin(\delta)\right)-2\delta(\cos(x)-1)\sin(\delta)+2\delta\sin(x)(1-\cos(\delta))$$ and finally \eqref{lemma:someestimates:iii} from the decomposition of $\vert e^{ix}-ix+x^2/2-e^{iy}+iy-y^2/2\vert^2$ into the following terms: 

\begin{enumerate}
\item $2(1-\cos(\delta))+\delta^2+\delta^4/4-2\delta\sin(\delta)-(1-\cos(\delta))\delta^2\leq \delta^6/36$ for any real $\delta$.
\item $\delta^3x-\sin(x)\sin(\delta)\delta^2=\delta^2(\delta(x-\sin(x))+\sin(x)(\delta-\sin(\delta)))\leq \vert\delta\vert^3\vert x\vert^3/6+\vert\delta\vert^5\vert x\vert/6$. 
\item $-2(1-\cos(\delta))\delta x+2\delta x(1-\cos(x))(1-\cos(\delta))+2\delta(1-\cos(\delta))\sin(x)=2\delta(1-\cos(\delta))(x(1-\cos(x))+\sin(x)-x)\leq \vert \delta\vert^3\vert x\vert^3/3$, since for all real $x$ one has $\vert \sin(x)-x\cos(x)\vert\leq \vert x\vert^3/3$.
\item $-(\cos(x)-1)(1-\cos(\delta))\delta^2\leq x^2\delta^4/4$.
\item $\delta^2x^2-2\delta x\sin(x)\sin(\delta)-2\delta\sin(\delta)(\cos(x)-1)=x^2\delta(\delta-\sin(\delta))+2\delta\sin(\delta)(1-\cos(x)-x\sin(x)+x^2/2)\leq \delta^4x^2/6+\delta^2x^4/4$ since for all real $x$, one has $0\leq 1-\cos(x)-x\sin(x)+x^2/2\leq x^4/8$.
\end{enumerate}
The latter inequalities are again seen to be true by comparing derivatives. 
\end{proof}

\begin{lemma}\label{lemma:multivariate}
Let $\{p,x,y\}\subset\mathbb{R}^d$. Then:
\begin{enumerate}
\item $\vert (e^{i\langle p,x\rangle}-1)-(e^{i\langle p,y\rangle}-1)\vert\leq \vert p\vert \vert x-y\vert$.
\item $\vert (e^{i\langle p,x\rangle}-i\langle p,x\rangle-1)-(e^{i\langle p,y\rangle}-i\langle p,y\rangle-1)\vert\leq 2\vert p\vert^2(\vert x\vert+\vert y\vert)\vert x-y\vert$. 
\end{enumerate}
\end{lemma}
\begin{proof}
This is an elementary consequence of the complex Mean Value Theorem \cite[p. 859, Theorem 2.2]{evard} and the Cauchy-Schwartz inequality. 
\end{proof}

\begin{lemma}\label{lemma:somethingsvnaish}
For any L\'evy measure $\lambda$ on $\mathbb{R}$, one has for the two functions (given for $1\geq \delta>0$): $M_0(\delta):=\delta^2\int_{[-1,1]\backslash (-\delta,\delta)}d\lambda(x)$ and $M_1(\delta):=\delta\int_{[-1,1]\backslash (-\delta,\delta)}\vert x\vert d\lambda(x)$ that $M_0(\delta)\rightarrow 0$ and $M_1(\delta)\rightarrow 0$ as $\delta\downarrow 0$. If, moreover, $\int_{[-1,1]}\vert x\vert d\lambda(x)<\infty$, then $\delta\int_{[-1,1]\backslash (-\delta,\delta)}d\lambda(x)\to 0$ as $\delta\downarrow 0$.  
\end{lemma}

\begin{proof}
Indeed let $\mu$ be the finite measure on $([-1,1],\mathcal{B}_{[-1,1]})$ given by $\mu(A):=\int_Ax^2d\lambda(x)$ ($A$ Borel subset of $[-1,1]$) and let
$f^0_\delta(x):=\left(\frac{\delta}{x}\right)^2\mathbbm{1}_{[-1,1]\backslash (-\delta,\delta)}(x)$ and $f^1_\delta(x):=\frac{\delta}{\vert x\vert}\mathbbm{1}_{[-1,1]\backslash (-\delta,\delta)}(x)$ be functions on $[-1,1]$. Clearly $0\leq f^0_\delta,f^1_\delta\leq 1$ and $f^0_\delta,f_\delta^1\rightarrow 0$ pointwise as $\delta\downarrow 0$. Hence by Lebesgue dominated convergence theorem (DCT), we have $M_0(\delta)=\int
f^0_\delta d\mu$ and $M_1(\delta)=\int f^1(\delta)d\mu$ converging to $\int 0d\mu=0$ as $\delta\downarrow 0$. The ``finite first absolute moment'' case is similar.
\end{proof}

\begin{proposition}\label{proposition:conditionQmatrix}
Under scheme 1, with $\sigma^2>0$, \eqref{eq:conditionsongenerator} holds for all sufficiently small $h$. Notation-wise, under either of the two schemes, we let $h_\star\in (0,+\infty]$ be such that $Q^h$ has non-negative off-diagonal entries for all $h\in (0,h_\star)$. 
\end{proposition}
\begin{proof}
If $V=0$ this is immediate. If $V=1$, then (via a triangle inequality):
\begin{eqnarray*}
h\vert \mu^h \vert&\leq&h\left\vert\sum_{s\in \Zh\backslash \{0\}}s\int_{A_s^h}\indd(y)d\lambda(y)\right\vert\leq h\sum_{s\in \Zh\backslash \{0\}}\int_{A_s^h}\vert s-u+u\vert\indd(y)d\lambda(y)\\
&\leq&h\left(\frac{h}{2}\lambda([-1,1]\backslash [-h/2,h/2])+\int_{[-1,1]\backslash [h/2,h/2]}\vert u\vert d\lambda(u)\right)\to 0
\end{eqnarray*}
as $h\downarrow 0$ by Lemma~\ref{lemma:somethingsvnaish}. Eventually the expression is smaller than $\sigma^2>0$ and  the claim follows. 
\end{proof}

Furthermore, we have the following inequalities, which together imply an estimate for $\vert \Psi^h-\Psi\vert$. In the following, recall the notation ($\delta\in (0, 1]$): $\zeta(\delta) := \delta\int_{[-1,1]\backslash [-\delta,\delta]} |x|d\lambda(x)$, $\gamma(\delta) := \delta^2\int_{[-1,1]\backslash [-\delta,\delta]} d\lambda(x)$, $c:=\lambda(\mathbb{R})$, $b := \kappa(0)$,  $d:=\lambda(\mathbb{R}\backslash [-1,1])$. Recall also the definition of the sets $A_s^h$ following \eqref{equation:infinitesimalgeneratorforY}.

\begin{proposition}[Convergence of characteristic exponents]\label{proposition:someestimates}
For all $p\in\mathbb{R}$: $0\leq f_h(p)\leq p^4h^2/4!$ and $0\leq i\sgn(p)g_h(p)\leq h^2\vert p\vert^3/3!$ (resp., under scheme 2, $\vert g_h(p)\vert\leq hp^2/2!$). Moreover:
\begin{enumerate}[(i)]
\item when $c<\infty$; with $V=0$: $\vert l_h(p)\vert\leq c\vert p\vert h/2$.
\item when $b<\infty=c$; with $V=1$; for all $h\leq 2$: $\vert l_h(p)\vert\leq \frac{h}{2}\left(\vert p\vert d+p^2b\right)+(p^2+\vert p\vert^3+p^4)o(h)$ (resp. under scheme 2, $\vert l_h(p)\vert\leq \frac{h}{2}\left(\vert p\vert d+2p^2b\right)+(p^2+\vert p\vert^3+p^4)o(h)$) where $o(h)$ depends only on $\lambda$.

\item when $b=\infty$; with $V=1$; for all $h\leq 2$: $\vert l_h(p)\vert\leq p^2\left(\zeta(h/2)+\frac{1}{2}\gamma(h/2)\right)+(\vert p\vert +\vert p\vert^3+p^4)O(h)$ (resp. under scheme 2, $\vert l_h(p)\vert\leq p^2\left[2\zeta(h/2)+\frac{1}{2}\gamma(h/2)\right]+(\vert p\vert +p^2 +\vert p\vert^3+p^4)O(h)$) where again $O(h)$ depends only on $\lambda$. Note here that we always have $\gamma\leq \zeta$ and that $\zeta$ decays strictly slower than $h$, as $h\downarrow 0$.
\end{enumerate}
\end{proposition}

\begin{remark}\label{remark:characteristic_exponents}
\begin{enumerate}[(i)]
\item \label{remark:characteristic_exponents:i} We may briefly summarize the essential findings of Proposition~\ref{proposition:someestimates} in Table~\ref{table:characteristic_exponent_convergence}, by noting that the following will have been proved for $p\in \mathbb{R}$ and $h\in (0,h_\star\land 2)$: 
\begin{equation}\label{estimate}
\vert \Psi^h(p)-\Psi(p)\vert\leq f(h)R(\vert p\vert)+o(f(h))Q(\vert p\vert)
\end{equation}
where $R$ and $Q$ are polynomials of respective degrees $\alpha$ and $\beta$ and $f:(0,h_\star\land 2)\to (0,\infty)$.
\begin{table}[!hbt]
\caption{Summary of Proposition~\ref{proposition:someestimates} via the triplet $(f(h),\alpha,\beta)$ introduced in \ref{remark:characteristic_exponents:i} of Remark~\ref{remark:characteristic_exponents}. We agree $\deg 0=-\infty$, where $0$ is the zero polynomial.}\label{table:characteristic_exponent_convergence}
\begin{center}

\begin{tabular}{|c|c|c|}\hline
$(f(h),\alpha,\beta)$ & $\sigma^2>0$ (scheme 1) & $\sigma^2=0$ (scheme 2)\\\hline
$\lambda(\mathbb{R})=0$ ($V=0$) & $(h^2,4,-\infty)$ & $(h,2,-\infty)$\\\hline
$\lambda(\mathbb{R})<\infty$ ($V=0$) & $(h,1,4)$ & $(h,2,-\infty)$\\\hline
$\kappa(0)<\infty=\lambda(\mathbb{R})$  ($V=1$)& \multicolumn{2}{|c|}{$(h,2,4)$}\\\hline
$\kappa(0)=\infty$ ($V=1$) & \multicolumn{2}{|c|}{$(\zeta(h/2),2,4)$}\\\hline
\end{tabular}

\end{center}
\end{table}

\item \label{remark:characteristic_exponents:ii} An analogue of \eqref{estimate} is got in the multivariate case, simply by looking directly at the difference of \eqref{eq:multivariate:exponent} and \eqref{eq:characteristic_exponent}. One does so either component by component (when it comes to the drift and diffusion terms), the estimates being then the same as in the univariate case; or else one employs, in addition, Lemma~\ref{lemma:multivariate} (for the part corresponding to the integral against the L\'evy measure). In  particular, \eqref{estimate} (with $p\in \mathbb{R}^d$) follows for suitable choices of $R$, $Q$ and $f$, and Table~\ref{table:characteristic_exponent_convergence} remains unaffected, apart from its last entry, wherein $\zeta$ should be replaced by $\zeta+\chi$ (one must also replace ``$\diffusion=0$'' (resp. ``$\diffusion>0$'') by ``$\Diffusion$ (resp. non-) degenerate''  (amalgamating scheme 1 \& 2 into the multivariate one) and $\measure(\mathbb{R})$ by $\Measure(\mathbb{R}^d)$). 

\item The above entails, in particular, convergence of $\Psi^h(p)$ to $\Psi(p)$ as $h\downarrow 0$ pointwise in $p\in\mathbb{R}$. L\'evy's continuity theorem \cite[p. 326]{dudley} and stationarity and independence of increments yield at once Remark~\ref{theorem:convergenceindistribution}. 
\item Note that we use $V=1$ rather than $V=0$ when $b<\infty=c$, because this choice yields linear convergence (locally uniformly) of $\Psi^h\to \Psi$. By contrast, retaining $V=0$, would have meant that the decay of $\Psi^h-\Psi$ would be governed, modulo terms which are $O(h)$, by the quantity $Q(h):=\sum_{s\in\Zh\backslash \{0\}}\int_{A_s^h\cap[-1,1]}(s-u)d\measure(u)$ (as will become clear from the estimates in the proof of Proposition~\ref{proposition:someestimates} below). But the latter can decay slower than $h$. In particular, consider the family of L\'evy measures, indexed by $\epsilon\in [0,1)$: $\measure_\epsilon=\sum_{n=1}^\infty w_n\delta_{-x_n}$, with $h_n=1/3^n$, $x_n=3h_n/2$, $w_n=1/x_n^\epsilon$, $n\geq 1$. For all these measures $b<\infty=c$. Furthermore, it is straightforward to verify that $\liminf_{n\to\infty}Q(h_n)/K(h_n)>0$, where $K(h)$ is $h^{1-\epsilon}$ or $h\log(1/h)$, according as $\epsilon\in (0,1)$ or $\epsilon=0$. 
\item \label{remark:characteristic_exponents:iv} It is seen from Table~\ref{table:characteristic_exponent_convergence} that the order of convergence goes from quadratic (at least when $\diffusion>0$) to  linear, to sublinear, according as the L\'evy measure is zero, $\measure(\mathbb{R})>0$ \& $\kappa(0)<\infty$, or $\kappa$ becomes more and more singular at the origin. Let us attempt to offer some intuition in this respect. First, the quadratic order of convergence is due to the convergence properties of the discrete second and symmetric first derivative. Further, as soon as the L\'evy measure is non-zero, the latter is aggregated over the intervals $(A_s^h)_{s\in\Zh\backslash \{0\}}$, length $h$, which (at least in the worst case scenario) commit respective errors of order $\measure(A_s^h)h$ or $\int_{A_s^h}(\vert x\vert\land 1)d\measure(x)h$ ($s \in \Zh\backslash \{0\}$) each, according as $V=0$ or $V=1$. Hence, the more singular the $\kappa$, the bigger the overall error. Figure~\ref{fig:comparison} depicts this progressive worsening of the convergence rate for the case of $\alpha$-stable L\'evy processes. 
\end{enumerate}
\end{remark}

\begin{figure}[!htb]
\includegraphics[width=\textwidth]{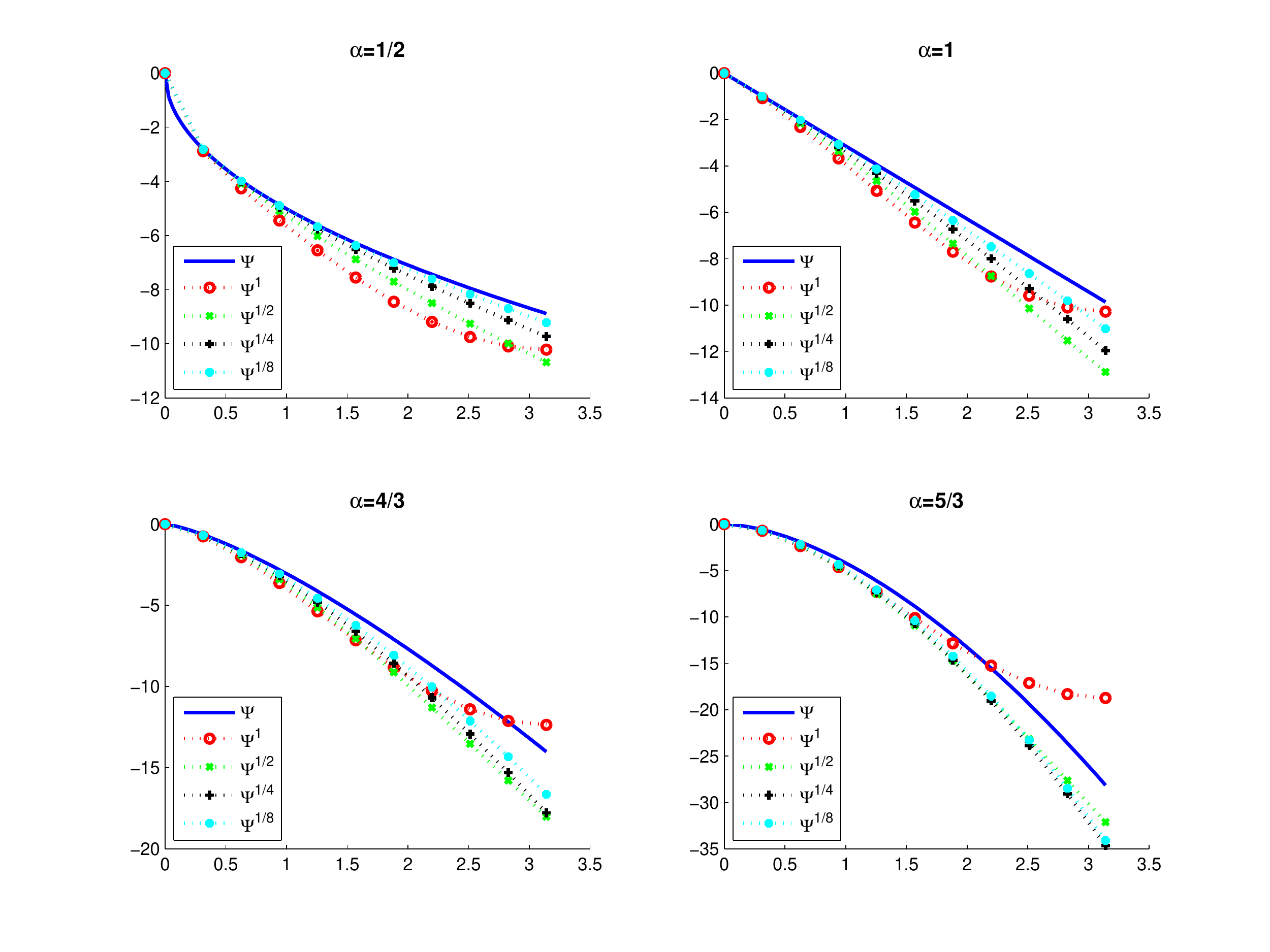}
\caption{Comparison of the convergence of characteristic exponents for $\alpha$-stable L\'evy processes, $\alpha\in \{1/2,1,4/3,5/3\}$; $\diffusion=0$, $\mu=0$ and $\measure(dx)=dx/\vert x\vert^{1+\alpha}$ (scheme 2, $V=1$). Each plot is of $\Psi$ and of $\Psi^h$ ($h\in \{1,1/2,1/4,1/8\}$) on the interval $[0,\pi]$. Note that $\kappa(0)=\infty$, precisely when $\alpha\geq 1$. The plots are indeed suggestive of a progressive worsening of the rate of convergence as $\alpha\uparrow$.}
\label{fig:comparison}
\end{figure}

\begin{proof} (Of Proposition~\ref{proposition:someestimates}.) The first two assertions are transparent by Lemma~\ref{lemma:someestimates} --- with the exception of the estimate under scheme 2, where (with $\delta:=hp$):
\begin{equation*}
\vert g_h(p)\vert=\frac{1}{h}\sqrt{\delta^2-2\delta\sin(\delta)+2(1-\cos(\delta))}\leq\frac{1}{h}\frac{\delta^2}{2}=hp^2/2!.
\end{equation*}

Further, if $c<\infty$ (under $V=0$):
\begin{eqnarray*}
\nonumber
&&\left \vert \sum_{s\in \Zh\backslash \{0\}}c_s^h(e^{isp}-1)-\int_{\mathbb{R}\backslash [-\frac{h}{2},\frac{h}{2}]}(e^{ipu}-1)d\lambda(u)\right\vert =\left\vert \sum_{s\in \Zh\backslash \{0\}}c_s^he^{isp}-\int_{\mathbb{R}\backslash [-\frac{h}{2},-\frac{h}{2}]}e^{ipu}d\lambda(u)\right\vert=\\\nonumber 
&=&\!\!\left\vert\sum_{s\in \Zh\backslash \{0\}}\int_{A_s^h}\left(e^{isp}-e^{ipu}\right)d\lambda(u)\right\vert\leq \sum_{s\in \Zh\backslash\{0\}}\int_{A_s^h}\left\vert 1-e^{ip(u-s)}\right\vert d\lambda(u)\leq \vert p\vert h\lambda\left(\mathbb{R}\backslash \left[-\frac{h}{2},\frac{h}{2}\right]\right)/2,
\end{eqnarray*}
where in the second inequality we apply \eqref{lemma:someestimates:i} of Lemma~\ref{lemma:someestimates} and the first follows from the triangle inequalities. Finally, $\vert \int_{[-h/2,h/2]}(e^{ipu}-1)d\lambda(u)\vert\leq \lambda([-h/2,h/2])\vert p\vert h/2$, again by \eqref{lemma:someestimates:i} of Lemma~\ref{lemma:someestimates}, and the claim follows. 

For the remaining two claims, in addition to recalling the general results of Lemma~\ref{lemma:someestimates}, we prepare the following specific estimates. Whenever $\{x,y\}\subset\mathbb{R}$, with $\delta:=y-x$, $0\ne \vert x\vert\geq \vert\delta\vert$, we have: 
\begin{itemize}
\item using the inequality $\sqrt{1+z}\leq 1+z/2$ ($z\geq 0$) and \eqref{lemma:someestimates:ii} of Lemma~\ref{lemma:someestimates}: 
\begin{equation}\label{eq:adiuvat}
\vert e^{ix}-ix-e^{iy}+iy\vert\leq\vert \delta x\vert\left(1+\frac{1}{2}\left\vert\frac{\delta}{x}\right\vert+\frac{1}{8}\frac{\delta^2}{x^2}\right)=\vert\delta x\vert+\frac{1}{2}\delta^2+\frac{1}{8}\left\vert\frac{\delta^3}{x}\right\vert\leq\vert \delta x\vert+\frac{5}{8}\delta^2.
\end{equation} 
\item using \eqref{lemma:someestimates:iii} of Lemma~\ref{lemma:someestimates}:
\begin{equation}\label{eq:adiuvatend}
\vert e^{ix}-ix -e^{iy}+iy\vert\leq \vert e^{ix}-e^{iy}-ix+iy+x^2/2-y^2/2\vert+\frac{1}{2}\vert x^2-y^2\vert\leq \frac{7}{6}\vert \delta\vert x^2+\vert \delta\vert\vert x\vert+\frac{1}{2}\delta^2.
\end{equation}
\end{itemize}

Now, when $c=\infty$ (under $V=1$; for all $h\leq 2$), denoting $\xi(\delta):=\int_{[-\delta,\delta]}x^2d\lambda(x)$, we have, under scheme 1, as follows:
\begin{equation}\label{eq:end:one}
\left\vert c^h_0\left(\frac{\cos(hp)-1}{h^2}+\frac{p^2}{2}\right)\right\vert\leq p^4h^2\xi(h/2)/4!.
\end{equation}

\begin{eqnarray}
\nonumber
&&\left\vert \int_{[-\frac{h}{2},\frac{h}{2}]} u^2\left(-\frac{p^2}{2}\right)d\lambda(u)-\int_{[-\frac{h}{2},\frac{h}{2}]}\left(e^{ipu}-1-ipu\right)d\lambda(u)\right\vert\\\nonumber
&\leq& \left\vert  \int_{[-\frac{h}{2},\frac{h}{2}]}\left(\cos(pu)-1+\frac{p^2u^2}{2!}\right)d\lambda(u)\right\vert+\left\vert \int_{[-\frac{h}{2},\frac{h}{2}]}\left(\sin(pu)-pu\right)d\lambda(u)\right\vert\\
&\leq&p^4(h/2)^2\xi(h/2)/4!+\vert p\vert^3(h/2)\xi(h/2)/3!.
\label{eq:end:two}
\end{eqnarray}

\begin{equation}\label{eq:end:three}
\vert-\mu^hg_h(p)\vert=\left\vert-i\mu^h\left(\frac{\sin(hp)}{h}-p\right)\right\vert\leq\frac{1}{3!}h^2\vert p\vert^3\left(\zeta(h/2)+\kappa(h/2)\right).
\end{equation}

\begin{eqnarray}
\nonumber
&&\left \vert \sum_{s\in \Zh\backslash \{0\}}c_s^h(e^{isp}-1)-ip\mu^h-\int_{\mathbb{R}\backslash [-\frac{h}{2},\frac{h}{2}]}(e^{ipu}-1-ipu\indd(u))d\lambda(u)\right\vert\\\nonumber
&\leq&\!\!\sum_{s\in \Zh\backslash \{0\}}\int_{A_s^h}\left\vert e^{ipu}-e^{ips}-ipu\indd(u)+ips\indd(u)\right\vert d\lambda(u)\\\nonumber
&\leq& \!\!\sum_{s\in \Zh\backslash \{0\}}\left[\int_{A_s^h\cap (\mathbb{R}\backslash [-1,1])}+\int_{A_s^h\cap [-1,1]}\right]\left\vert e^{ipu}-e^{ips}-ipu\indd(u)+ips\indd(u)\right\vert d\lambda(u)\\
&\leq&\!\! \frac{h}{2}\vert p\vert \int_{\mathbb{R}\backslash [-1,1]}d\lambda(u)+p^2\frac{h}{2}\int_{[-1,1]\backslash [-\frac{h}{2},\frac{h}{2}]}\!\!\vert u\vert d\lambda(u)+ p^2\frac{5}{8}\left(\frac{h}{2}\right)^2\lambda([-1,1]\backslash [-h/2,h/2]),
\label{eq:end:four}
\end{eqnarray}
where, in particular, we have applied \eqref{eq:adiuvat} to $x=ps$, $y=pu$. If in addition $b=\infty$, we opt rather to use \eqref{eq:adiuvatend}, again with $x=ps$ and $y=pu$, and obtain instead: 
\begin{eqnarray}
\nonumber
&&\left \vert \sum_{s\in \Zh\backslash \{0\}}c_s^h(e^{isp}-1)-ip\mu^h-\int_{\mathbb{R}\backslash [-\frac{h}{2},\frac{h}{2}]}(e^{ipu}-1-ipu\indd(u))d\lambda(u)\right\vert\\\nonumber
&\leq& \frac{h}{2}\vert p\vert \int_{\mathbb{R}\backslash [-1,1]}d\lambda(u)+p^2\frac{h}{2}\int_{[-1,1]\backslash [-\frac{h}{2},\frac{h}{2}]}\vert u\vert d\lambda(u)+ p^2\frac{1}{2}\left(\frac{h}{2}\right)^2\lambda([-1,1]\backslash [-h/2,h/2])+\\
&+&\frac{7}{6}\vert p\vert^3\frac{h}{2}\int_{[-1,1]}x^2d\measure(x).
\label{eq:end:four:prime}
\end{eqnarray}
Under scheme 2, \eqref{eq:end:one}, \eqref{eq:end:two} and \eqref{eq:end:four}/\eqref{eq:end:four:prime} remain unchanged, whereas \eqref{eq:end:three} reads: 
\begin{equation}\label{eq:end2:three} 
\left\vert\mu^hg_h(p)\right\vert\leq \frac{h}{2}p^2\left(\zeta(h/2)+\kappa(h/2)\right).
\end{equation}
Now, combining \eqref{eq:end:one}, \eqref{eq:end:two}, \eqref{eq:end:three} and \eqref{eq:end:four} under scheme 1 (resp. \eqref{eq:end:one}, \eqref{eq:end:two}, \eqref{eq:end2:three} and \eqref{eq:end:four}  under scheme 2), yields the desired inequalities when $b<\infty$. If $b=\infty$ use \eqref{eq:end:four:prime} in place of \eqref{eq:end:four}. 
\end{proof}

\section{Rates of convergence for transition kernels}\label{section:Rates}
Finally let us incorporate the estimates of Proposition~\ref{proposition:someestimates} into an estimate of $D^h_{t,T}(x,y)$ (recall the notation in \eqref{eq:fundamental_notation}). Assumption~\ref{assumption} and Table~\ref{table:reference} are understood as being in effect throughout this section from this point onwards. Recall that $\vert\Psi^h-\Psi\vert\leq \sigma^2\vert f_h \vert+\mu \vert g_h\vert+\vert l_h\vert$ and that the approximation is considered for $h\in (0,h_\star)$ (cf. Proposition~\ref{proposition:conditionQmatrix}).

First, the following observation, which is a consequence of the $h$-uniform growth of $-\Re\Psi^h(p)$ as $\vert p\vert\to\infty$, will be crucial to our endeavour (compare Remark~\ref{proposition:levydensity}).

\begin{proposition}\label{proposition:uniform_coercivity}
For some $\{P,C,\epsilon\}\subset (0,\infty)$ and $h_0\in (0,h_\star]$, depending only on $\{\lambda,\sigma^2\}$, and then all $h\in (0,h_0)$, $p\in [-\pi/h,\pi/h]\backslash (-P,P)$ and $t\geq 0$: $\vert \phi_{X_t}^h(p)\vert\leq \exp\{-Ct\vert p\vert^\epsilon\}$. Moreover, when $\sigma^2>0$, we may take $\epsilon=2$, $P=0$, $C=\frac{1}{2}\left(\frac{2}{\pi}\right)^2$ and $h_0=h_\star$, whereas otherwise $\epsilon$ may take the same value as in Orey's condition (cf.  Assumption~\ref{assumption}). 
\end{proposition}
\begin{proof}
Assume first $\diffusion>0$, so that we are working under scheme 1. It is then clear from \eqref{eq:spectralthingy2} that: $$-\Re\Psi^h(p)\geq \diffusion \frac{1-\cos(hp)}{h^2}\geq \frac{1}{2}\left(\frac{2}{\pi}\right)^2\sigma^2p^2,$$ since $1-\cos(x)=2\sin^2(x/2)\geq 2\left(\frac{x}{\pi}\right)^2$  for all $x\in [-\pi,\pi]$. 
On the other hand, if $\diffusion=0$, we work under scheme 2 and necessarily $V=1$. In that case it follows from \eqref{eq:spectralthingy} for $h\leq 2$ and $p\in [-\pi/h,\pi/h]\backslash \{0\}$, that: 
\begin{eqnarray*}
-\Re\Psi^h(p)&\geq& \left[c_0^h\frac{1-\cos(hp)}{h^2}+\sum_{s\in \Zh\backslash \{0\}}c_s^h(1-\cos(sp))\right]\\
&\geq&\frac{2}{\pi^2}p^2\left(\int_{A_0^h}u^2d\measure(u)+\sum_{s\in \Zh\backslash \{0\},\vert s\vert\leq  \frac{\pi}{\vert p\vert}}s^2c_s^h\right)\\
&\geq&\frac{2}{\pi^2}p^2\left(\int_{A_0^h}u^2d\measure(u)+\frac{4}{9}\sum_{s\in \Zh\backslash \{0\},\vert s\vert\leq  \frac{\pi}{\vert p\vert}}\int_{A_s^h}u^2d\measure(u)\right)\\
&\geq&\frac{2}{\pi^2}p^2\left(\int_{A_0^h}u^2d\measure(u)+\frac{4}{9}\int_{[-\left(\frac{\pi}{\vert p\vert}-\frac{h}{2}\right),\frac{\pi}{\vert p\vert}-\frac{h}{2}]\backslash A_0^h}u^2d\measure(u)\right)\\
&\geq& \frac{8}{9}\frac{1}{\pi^2}p^2 \int_{\left[-\left(\left(\frac{\pi}{\vert p\vert}-\frac{h}{2}\right)\lor \frac{h}{2}\right),\left(\left(\frac{\pi}{\vert p\vert}-\frac{h}{2}\right)\lor \frac{h}{2}\right)\right]}u^2d\measure(u)\\
&\geq&  \frac{8}{9}\frac{1}{\pi^2}p^2 \int_{[-\frac{1}{2}\frac{\pi}{\vert p\vert},\frac{1}{2}\frac{\pi}{\vert p\vert}]}u^2d\measure(u).
\end{eqnarray*}
Now invoke Assumption~\ref{assumption}. There are some $\{r_0,A_0\}\in (0,+\infty)$ such that for all $r\in (0,r_0]$: $\int_{[-r,r]}u^2d\measure(u)\geq A_0r^{2-\epsilon}$. Thus for $P=\pi/(2r_0)$ and then all $p\in \mathbb{R}\backslash (-P,P)$, we obtain: $$\int_{[-\frac{1}{2}\frac{\pi}{\vert p\vert},\frac{1}{2}\frac{\pi}{\vert p\vert}]}u^2d\measure(u)\geq A_0\left(\frac{1}{2}\frac{\pi}{\vert p\vert}\right)^{2-\epsilon},$$ from which the desired conclusion follows at once. Remark that, possibly, $r_0$ may be taken as $+\infty$, in which case $P$ may be taken as zero.
\end{proof}
Second, we have the following general observation which concerns the transfer of the rate of convergence from the characteristic exponents to the transition kernels. Its validity is in fact independent of Assumption~\ref{assumption}.

\begin{proposition}\label{proposition:transfer}
Suppose there are $\{P,C,\epsilon\}\subset (0,\infty)$, a real-valued polynomial $R$, an $h_0\in (0,h_\star]$, and a function $f:(0,h_0)\to (0,\infty)$, decaying to $0$ no faster than some power law, such that for all $h\in (0,h_0)$:
\begin{enumerate}
\item \label{transfer:one} for all $p\in [-\pi/h,\pi/h]$: $\vert \Psi^h(p)-\Psi(p)\vert\leq f(h)R(\vert p\vert)$.
\item\label{transfer:two} for all $s>0$ and $p\in [-\pi/h,\pi/h]\backslash (-P,P)$: $\vert \phi_{X_s^h}(p)\vert\leq \exp\{-Cs\vert p\vert^\epsilon\}$; whereas for $p\in \mathbb{R}\backslash (-P,P)$: $\vert \phi_{X_s}(p)\vert\leq \exp\{-Cs\vert p\vert^\epsilon\}$.
\end{enumerate}
Then for any $s>0$, $\Delta_s(h)=O(f(h))$. 
\end{proposition}
Before proceeding to the proof of this proposition, we note explicitly the following elementary, but key lemma:

\begin{lemma}\label{lemma:keylemma}
For  $\{z,v\}\subset \mathbb{C}$: $\vert e^{z}-e^{v}\vert\leq (\vert e^{z}\vert \lor \vert  e^{v}\vert)\vert z-v\vert$.
\end{lemma}
\begin{proof}
This follows from the inequality $\vert e^z-1\vert\leq\vert z\vert$ for $\Re z\leq 0$, whose validity may be seen by direct estimation. 
\end{proof}

\begin{proof}(Of Proposition~\ref{proposition:transfer}.)
From \eqref{equation:trnasitionkernelcont} and \eqref{equation:trnasitionkerneldiscrete} we have for the quantity $\Delta_s(h)$ from \eqref{eq:fundamental_notation}: 
$$\Delta_s(h)\leq \int_{\mathbb{R}\backslash (-\pi/h,\pi/h)} \vert \exp\{\Psi(p)s\}\vert dp+\int_{[-\pi/h,\pi/h]}\vert \exp\{\Psi^h(p)s\}-\exp\{\Psi(p)s\}\vert dp.$$ Then the first term decays faster than any power law in $h$ by \eqref{transfer:two} and L'H\^opital's rule, say, while in the second term we use the estimate of Lemma~\ref{lemma:keylemma}. Since $\exp\{-Ct\vert p\vert^\epsilon\}dp$ integrates every polynomial in $\vert p\vert$ absolutely, by \eqref{transfer:one} and \eqref{transfer:two} integration in the second term can then be extended to the whole of $\mathbb{R}$ and the claim follows. 
\end{proof}
This last proposition allows us to transfer the rates of convergence directly from those of the characteristic exponents to the transition kernels. In particular, Theorem~\ref{assumption:multivariate} follows from a straightforward extension (of the proof) of Proposition~\ref{proposition:transfer} to the multivariate setting, \ref{remark:characteristic_exponents:ii} of Remark~\ref{remark:characteristic_exponents}, Assumption~\ref{assumption:multivariate} and Remark~\ref{remark:multiariate:condition_discussion}. Returning to the  univariate case, analogous conclusions could be got from Remark~\ref{proposition:levydensity}, Proposition~\ref{proposition:uniform_coercivity} (themselves both consequences of Assumption~\ref{assumption}) and Proposition~\ref{proposition:someestimates}. In the sequel, however, in the case when $\diffusion>0$, we shall be interested in a more precise estimate of the constant in front of the leading order term ($D_1$ in the statement of Theorem~\ref{assumption:multivariate}). Moreover, we shall want to show our estimates are tight in an appropriate precise sense. 

To this end we assume given a function $K$ with the properties that: 
\begin{enumerate}[label=(F),ref=(F)]
\item \label{(F)} $0\leq K(h)\to \infty$ as $h\downarrow 0$ and $K(h)\leq \frac{\pi}{h}$ for all sufficiently small $h$;
\end{enumerate}
\begin{enumerate}[label=(E),ref=(E)]
\item \label{(E)} the quantity $$\mathcal{A}(h):=\left[\int_{-\infty}^{-K(h)}+\int_{K(h)}^\infty\right]\left\vert \exp\{\Psi(p)s\}\right\vert dp+\left[\int_{-\frac{\pi}{h}}^{-K(h)}+\int_{K(h)}^{\frac{\pi}{h}}\right]\left\vert \exp\{\Psi^h(p)s\}\right\vert dp$$ decays faster than the leading order term in the estimate of $D_{t,T}^h(x,y)$ (for which see, e.g., Table~\ref{table:characteristic_exponent_convergence});
\end{enumerate}
\begin{enumerate}[label=(C),ref=(C)]
\item  \label{(C)} $\sup_{[-K(h),K(h)]}\vert\Psi^h-\Psi\vert\leq 1$ for all small enough $h$.
\end{enumerate}
(suitable choices of $K$ will be identified later, cf. Table~\ref{table:choice_of_K(h)}). We now comment on the reasons behind these choices. 

First, the constants $\{C,P,\epsilon\}$ are taken so as to satisfy simultaneously Remark~\ref{proposition:levydensity} and Proposition~\ref{proposition:someestimates}. In particular, if $\diffusion>0$, we take $\epsilon=2$, $P=0$, $C=\frac{1}{2}\diffusion$, and if $\diffusion=0$, we may take $\epsilon$ from Orey's condition (cf. Assumption~\ref{assumption}). 

Next, we divide the integration regions in (\ref{equation:trnasitionkernelcont}) and (\ref{equation:trnasitionkerneldiscrete}) into five parts (cf. property \ref{(F)}): $(-\infty,-\frac{\pi}{h}],\ (-\frac{\pi}{h},-K(h)),\ [-K(h),K(h)],\ (K(h),\frac{\pi}{h}),\ [\frac{\pi}{h},\infty)$. Then we separate (via a triangle inequality) the integrals in the difference $D_{t,T}^h(x,y)$ accordingly and use the triangle inequality in the second and fourth region, thus (with $s:=T-t>0$):
\footnotesize
\begin{eqnarray*}
2\pi D_{t,T}^h(x,y)&\leq& \left[\int_{-\infty}^{-\pi/h}+\int_{\pi/h}^\infty\right]\left\vert \exp\{\Psi(p)s\}\right\vert dp+\left[\int_{-\frac{\pi}{h}}^{-K(h)}+\int_{K(h)}^{\frac{\pi}{h}}\right]\left(\left\vert \exp\{\Psi^h(p)s\}\right\vert +\left\vert \exp\{\Psi(p)s\}\right\vert\right) dp+\\
&+&\int_{-K(h)}^{K(h)}\left\vert \exp\{\Psi(p)s\}- \exp\{\Psi^h(p)s\}\right\vert dp.
\end{eqnarray*}
\normalsize
Finally, we gather the terms with $\left\vert \exp\{\Psi(p)s\}\right\vert$ in the integrand and use $\vert e^z-1\vert\leq e^{\vert z\vert}-1$ ($z\in\mathbb{C}$) to estimate the integral over $[-K(h),K(h)]$, so as to arrive at:
\begin{equation}
2\pi D_{t,T}^h(x,y)\leq \mathcal{A}(h)+\int_{-K(h)}^{K(h)}\vert \exp\{\Psi(p)s\}\vert\left( \exp\left\{s\left\vert \Psi^h(p)-\Psi(p)\right\vert\right\}-1\right) dp.\label{part:interior}
\end{equation}
Now, the rate of decay of $\mathcal{A}(h)$ can be controlled by choosing $K(h)$ converging to $+\infty$ fast enough, viz. property \ref{(E)}. On the other hand, in order to control the second term on the right-hand side of the inequality in \eqref{part:interior}, we choose $K(h)$ converging to $+\infty$ slowly enough so as to guarantee \ref{(C)}. Table~\ref{table:choice_of_K(h)} lists suitable choices of $K(h)$. It is easily checked from Table~\ref{table:characteristic_exponent_convergence} (resp. using L'H\^opital's rule coupled with Remark~\ref{proposition:levydensity} and Proposition~\ref{proposition:uniform_coercivity}), that these choices of $K(h)$ do indeed satisfy \ref{(C)} (resp. \ref{(E)}) above. Property \ref{(F)} is straightforward to verify. 

\begin{table}[!hbt]
\caption{Suitable choices of $K(h)$. For example, the $\diffusion>0$ and $\lambda(\mathbb{R})=0$ entry indicates that we choose $K(h)=\sqrt{\frac{2}{Cs}\log \frac{1}{h}}$ and then $\mathcal{A}(h)$ is of order $o(h^2)$.}\label{table:choice_of_K(h)}
\begin{center}
{
\renewcommand{\arraystretch}{1.5}
\begin{tabular}{|c|c|c|}\hline
& $\sigma^2>0$ (scheme 1) & $\sigma^2=0$ (scheme 2)\\\hline
$\lambda(\mathbb{R})=0$ ($V=0$) & $K(h)=\sqrt{\frac{2}{Cs}\log \frac{1}{h}}$ $\rightarrow$ $\mathcal{A}(h)=o(h^2)$  & $\times$\\\hline
$\lambda(\mathbb{R})<\infty$ ($V=0$) & $K(h)=\sqrt{\frac{1}{Cs}\log \frac{1}{h}}$ $\rightarrow$ $\mathcal{A}(h)=o(h)$  & $\times$\\\hline
$\kappa(0)<\infty=\lambda(\mathbb{R})$  ($V=1$)&  $K(h)=\sqrt{\frac{1}{Cs}\log \frac{1}{h}}$ $\rightarrow$ $\mathcal{A}(h)=o(h)$ & $K(h)=\sqrt[\epsilon]{\frac{2}{Cs}\log \frac{1}{h}}$ $\rightarrow$ $\mathcal{A}(h)=o(h)$ \\\hline
$\kappa(0)=\infty$ ($V=1$) & \multicolumn{2}{|c|}{$K(h)=\left(\frac{1}{\zeta(h/2)}\right)^{1/4}$ $\rightarrow$ $\mathcal{A}(h)=o(\zeta(h/2))$} \\\hline
\end{tabular}
}
\end{center}
\end{table}
Further, due to \ref{(C)}, for all sufficiently small $h$, everywhere on $[-K(h),K(h)]$: 
\small $$e^{s\vert \Psi^h-\Psi\vert}-1=s\vert \Psi^h-\Psi\vert+\sum_{k=2}^\infty\frac{(s\vert \Psi^h-\Psi\vert)^k}{k!}\leq s\vert \Psi^h-\Psi\vert+(s\vert \Psi^h-\Psi\vert)^2e^{s\vert \Psi^h-\Psi\vert}\leq s\vert \Psi^h-\Psi\vert+e(s\vert \Psi^h-\Psi\vert)^2.$$ \normalsize
Manifestly the second term will always decay strictly faster than the first (so long as they are not $0$). Moreover, since $ \exp\{-Cs\vert p\vert^\epsilon\}dp$ integrates every polynomial in $\vert p\vert$ (cf. the findings of Proposition~\ref{proposition:someestimates}) absolutely, it will therefore be sufficient in the sequel to estimate (cf. \eqref{part:interior}): 
\begin{equation}\label{equation:estimate}
\frac{s}{2\pi}\int_\mathbb{R} \exp\{-Cs\vert p\vert^\epsilon\}\left\vert \Psi^h(p)-\Psi(p)\right\vert dp.
\end{equation}

On the other hand, for the purposes of establishing sharpness of the rates for the quantity $D^h_{t,T}(x,y)$, we make the following: 

\begin{remark}[RD]\label{remark:RD}
Suppose we seek to prove that $f\geq 0$ converges to $0$ no faster than $g>0$, i.e. that $\limsup_{h\downarrow 0}f(h)/g(h)\geq C>0$ for some $C$. If one can show $f(h)\geq A(h)-B(h)$ and $B=o(g)$, then to show $\limsup_{h\downarrow 0}f(h)/g(h)\geq C$, it is sufficient to establish $\limsup_{h\downarrow 0}A(h)/g(h)\geq C$. We refer to this extremely useful principle as \emph{reduction by domination} (henceforth RD). 
\end{remark}

In particular, it follows from our above discussion, that it will be sufficient to consider (we shall always choose $x=y=0$):
\begin{equation}\label{equation:tightness}
\frac{s}{2\pi}\int_{-K(h)}^{K(h)} e^{s\Psi(p)}\left( \Psi^h(p)-\Psi(p)\right) dp,
\end{equation}
i.e. in Remark~\ref{remark:RD} this is $A$, and the difference to $D_{t,T}(0,0)$ represents $B$. Moreover, we can further replace $\Psi^h(p)-\Psi(p)$ in the integrand of \eqref{equation:tightness} by any expression whose difference to $\Psi^h(p)-\Psi(p)$ decays, upon integration, faster than the leading order term. For the latter reductions we (shall) refer to the proof of Proposition~\ref{proposition:someestimates}.

We have now brought the general discussion as far as we could. The rest of the analysis must invariably deal with each of the particular instances separately and we do so in the following two propositions. Notation-wise we let DCT stand for Lebesgue dominated convergence theorem.

\begin{proposition}[Convergence of transition kernels --- $\diffusion>0$]\label{propositon:convergence_kernels_diffusion_positive}
Suppose $\diffusion>0$. Then for any $s=T-t>0$: 
\begin{enumerate}
\item if $\measure(\mathbb{R})=0$: 
\begin{equation*}\label{eq:BM_and_drift_final}
\Delta_s(h)\leq h^2\left[\frac{1}{3\pi}\frac{\vert\mu\vert}{\sigma^4s}+\frac{1}{8\sqrt{2\pi}}\frac{1}{(s\sigma^2)^{3/2}}\right]+o(h^2).
\end{equation*}
Moreover, with $\sigma^2s=1$ and $\mu=0$ we have $\limsup_{h\downarrow 0}D_{t,T}^h(0,0)/h^2\geq 1/(8\sqrt{2\pi})$, proving that in general the convergence rate is no better than quadratic.
\item if $0<\lambda(\mathbb{R})<\infty$:
\begin{equation*}\label{eq:CP_and_diffusion_final}
\Delta_s(h)\leq h\frac{1}{2\pi}\frac{c}{\sigma^2}+o(h).
\end{equation*}
Moreover, with $\sigma^2=s=1$, $\mu=0$ and $\lambda=\frac{1}{2}(\delta_{1/2}+\delta_{-1/2})$ one has $\limsup_{h\downarrow 0}D_{t,T}^h(0,0)/h>0$ showing that convergence in general is indeed no better than linear.
\item if $\kappa(0)<\infty=\lambda(\mathbb{R})$:
\begin{equation*}\label{eq:finite_varuation_final}
\Delta_s(h)\leq h\left[\frac{1}{2\pi}\frac{d}{\sigma^2}+\frac{1}{2\sqrt{2\pi}}\frac{bs}{(\sigma^2s)^{3/2}}\right]+o(h).
\end{equation*}
Moreover, with $\sigma^2=s=1$, $\mu=0$ and $\lambda=\frac{1}{2}(\delta_{3/2}+\delta_{-3/2})+\frac{1}{2}\sum_{k=1}^\infty(\delta_{1/3^k}+\delta_{-1/3^k})$, one has $\limsup_{h\downarrow 0}D_{t,T}^h(0,0)/h>0$. 
\item if $\kappa(0)=\infty$:
\begin{equation*}
\Delta_s(h)\leq \frac{1}{\sqrt{2\pi}}\frac{s}{(\sigma^2s)^{3/2}}\left(\zeta(h/2)+\frac{1}{2}\gamma(h/2)\right)+o(\zeta(h/2)).
\end{equation*}
Moreover, with $\sigma^2=s=1$, $\mu=0$, and $\lambda=\sum_{k=1}^\infty w_k(\delta_{x_k}+\delta_{-x_k})$, where $x_n=\frac{3}{2}\frac{1}{3^n}$ and $w_n=1/x_n$ ($n\in\mathbb{N}$), one has $\limsup_{h\downarrow 0}D_{t,T}^h(0,0)/\zeta(h/2)>0$.
\end{enumerate}
\end{proposition}
\begin{proof}
Estimates of $\Delta_s(h)$ follow at once from \eqref{equation:estimate} and Proposition~\ref{proposition:someestimates}, simply by integration. As regards establishing sharpness of the estimates, however, we have as follows (recall that we always take $x=y=0$):
\begin{enumerate}
\item $\measure(\mathbb{R})=0$. Using \eqref{equation:tightness} it is sufficient to consider: $$A(h):=\frac{1}{2\pi}\left\vert \int_{-K(h)}^{K(h)}\exp\left\{-\frac{1}{2}p^2\right\}f_h(p)dp\right\vert.$$ By DCT, we have $A(h)/h^2\to \frac{1}{2\pi}\int_{-\infty}^\infty \exp\{-\frac{1}{2}p^2\}p^4/4!dp$ and the claim follows. 
\item  $0<\lambda(\mathbb{R})<\infty$.  Using \eqref{equation:tightness} and further RD via the estimates in the proof of Proposition~\ref{proposition:someestimates}, we conclude that it is sufficient to observe for the sequence $(h_n=\frac{1}{3^n})_{n\geq 1}\downarrow 0$ that: $$\limsup_{n\to\infty}\frac{1}{2\pi h_n}\left\vert \int_{-K(h_n)}^{K(h_n)}\exp\left\{-\frac{1}{2}p^2-1+\cos(p/2)\right\}l_{h_n}(p)dp\right\vert>0.$$ It is also clear that we may express: \footnotesize $$l_{h_n}(p)=2\frac{1}{2}\Re \left(e^{ip(1/2-h_n/2)}-e^{ip/2}\right)=\cos(p/2)(\cos(ph_n/2)-1)+\sin(p/2)\sin(ph_n/2).$$ \normalsize Therefore, by further RD, it will be sufficient to consider:  $$\limsup_{n\to\infty}\frac{1}{2\pi h_n}\left\vert \int_{-K(h_n)}^{K(h_n)}\exp\left\{-\frac{1}{2}p^2-1+\cos(p/2)\right\}\sin(ph_n/2)\sin(p/2)dp\right\vert.$$ By DCT it is equal to: $$I:=\frac{1}{2\pi}\int_0^\infty p\sin(p/2)\exp\{-\frac{1}{2}p^2-1+\cos(p/2)\}dp.$$ 
The numerical value of this integral is (to one decimal place in units of $2\pi$) $0.4/(2\pi)$, but we can show that $I>0$ analytically. Indeed the integrand is positive on $[0,6]$. Hence $2\pi eI\geq \sin(1/2)e^{\cos(3/2)}\int_1^3pe^{-p^2/2}dp-e\int_6^\infty pe^{-p^2/2}dp=\sin(1/2)e^{\cos(3/2)}[e^{-1/2}-e^{-9/2}]-e^{-17}$. Now use $\sin(1/2)\geq (1/2)\cdot (2/\pi)$ (which follows from the concavity of $\sin$ on $[0,\pi/2]$), so that, very crudely: $2\pi eI\geq (1/\pi)e^{-1/2}(1-e^{-4})-e^{-17}\geq (1/\pi)e^{-1/2}(1/2)-e^{-17}\geq (1/e^2)e^{-1/2}(1/e)-e^{-17}\geq e^{-4}-e^{-17}>0$. 
\item  $\kappa(0)<\infty=\lambda(\mathbb{R})$. Let $h_n=1/3^n$, $n\geq 1$. Because the second term in $\lambda$ lives on $\cup_{n\in\mathbb{N}}\mathbb {Z}_{h_n}$, it is seen quickly (via RD) that one need only consider (to within non-zero multiplicative constants): \footnotesize $$\limsup_{n\to\infty}\int_{-K(h_n)}^{K(h_n)}\frac{1}{h_n}\sin(ph_n/2)\sin (3p/2)\exp \left\{-\frac{1}{2}p^2+(\cos(3p/2)-1)+\sum_{k=1}^\infty(\cos(p/3^k)-1)\right\}dp.$$ \normalsize By DCT it is sufficient to observe that: \footnotesize $$\int_0^{2\pi/3}\sin(3p/2)p\exp\left\{-\frac{1}{2}p^2+(\cos(3p/2)-1)-\frac{p^2}{2}\sum_{k=1}^\infty\frac{1}{9^k}\right\}dp-\int_{2\pi/3}^\infty p\exp\left\{-\frac{1}{2}p^2\right\}dp>0.$$ \normalsize To see the latter, note that the second integral is immediate and equal to: $e^{-(2\pi/3)^2/2}$. As for the first one, make the change of variables $u=3p/2$. Thus we need to establish that: $$A:=(4/(9e))\int_0^\pi\sin(u)u\exp\{-u^2/4+\cos(u)\}du-e^{-(2\pi/3)^2/2}>0.$$ Next note that $-u^2/4+\cos(u)$ is decreasing on $[0,\pi]$ and the integrand in $A$ is positive. It follows that: 
\footnotesize
$$A\geq \frac{4}{9e}\left[\int_0^{\pi/3}u\sin(u)\exp\left\{-\frac{1}{4}\left(\frac{\pi}{3}\right)^2+\cos\left(\frac{\pi}{3}\right)\right\}du+\int_{\pi/3}^{\pi/2}u\sin(u)\exp\left\{-\frac{1}{4}\left(\frac{\pi}{2}\right)^2+\cos\left(\frac{\pi}{2}\right)\right\}du\right]-e^{-2\pi^2/9}.$$ 
\normalsize
Using integration by parts, it is now clear that this expression is algebraic over the rationals in $e$, $\sqrt{3}$ and the values of the exponential function at rational multiples of $\pi^2$. Since this explicit expression can be estimated from below by a positive quantity, 
one can check that $A>0$.
\item\label{convergence:proof:infinite_variation}$\kappa(0)=\infty$. Let again $h_n=1/3^n$, $n\geq 1$. Notice that: \footnotesize $$\sigma_1:=\int_{[-1,1]\backslash [-h_n/2,h_n/2]} u^2d\lambda(u)=2\sum_{k=1}^nx_k^2w_k,\text{ and }\sigma_2:=\sum_{s\in S_{h_n}\backslash \{0\}}c_s^{h_n}s^2=2\sum_{k=1}^n\left(x_k-\frac{h_n}{2}\right)^2w_k,$$ \normalsize so that $\Delta:=\sigma_1-\sigma_2=2\zeta(h_n/2)-\gamma(h_n/2)\geq \zeta(h_n/2)$. Using \eqref{lemma:someestimates:iii} of Lemma~\ref{lemma:someestimates} in the estimates of Proposition~\ref{proposition:someestimates}, it is then not too difficult to see that it is sufficient to show $\int_{-K(h_n)}^{K(h_n)}p^2\exp\{\Psi(p)\}dp$ converges to a strictly positive value as $n\to\infty$, which is transparent (since $\Psi$ is real valued). 
\end{enumerate}
\end{proof}

\begin{proposition}[Convergence of transition kernels --- $\diffusion=0$]\label{propositon:convergence_kernels_diffusion_zero}
Suppose $\diffusion=0$. For any $s=T-t>0$: 
\begin{enumerate}
\item if Orey's condition is satisfied and $\kappa(0)<\infty=\lambda(\mathbb{R})$, then $\Delta_s(h)=O(h)$. Moreover, with $\sigma^2=0$, $s=1$, $\mu=0$ and $\lambda=\frac{1}{2}\sum_{k=1}^\infty w_k(\delta_{x_k}+\delta_{-x_k})$, where, $x_n=\frac{3}{2}\frac{1}{3^n}$ and $w_n=1/\sqrt{x_n}$ ($n\in\mathbb{N}$), Orey's condition holds with $\epsilon=1/2$ and one has $\limsup_{h\downarrow 0}D_{t,T}^h(0,0)/h>0$. 
\item if Orey's condition is satisfied and $\kappa(0)=\infty$, then $\Delta_s(h)=O(\zeta(h/2))$. Moreover, with $\sigma^2=0$, $s=1$, $\mu=0$, and $\lambda=\sum_{k=1}^\infty w_k(\delta_{x_k}+\delta_{-x_k})$, where $x_n=\frac{3}{2}\frac{1}{3^n}$ and $w_n=1/x_n$ ($n\in\mathbb{N}$),  Orey's condition holds with $\epsilon=1$ and one has $\limsup_{h\downarrow 0}D_{t,T}^h(0,0)/\zeta(h/2)>0$.
\end{enumerate}
\end{proposition}
\begin{proof}
Again the rates of convergence for $\Delta_s(h)$ follow at once from \eqref{equation:estimate} and Proposition~\ref{proposition:someestimates} (or, indeed, from Proposition~\ref{proposition:transfer}). As regards sharpness of these rates, we have  (recall that we take $x=y=0$): 
\begin{enumerate}
\item $\kappa(0)<\infty=\lambda(\mathbb{R})$. Let $h_n=1/3^n$, $n\geq 1$. By checking Orey's condition on the decreasing sequence $(h_n)_{n\geq 1}$, Assumption~\ref{assumption} is satisfied with $\epsilon=1/2$ and we have $b<\infty=c$. $\mu^h=0$ by symmetry. Moreover by \eqref{equation:tightness}, and by further going through the estimates of Proposition~\ref{proposition:someestimates} using RD, it suffices to show: $$\limsup_{n\to\infty}\frac{1}{h_n}\left\vert \int_{-K(h_n)}^{K(h_n)}\exp\{\Psi(p)\}\left(\sum_{s\in \mathbb{Z}_{h_n}\backslash \{0\}}\int_{A_s^{h_n}}\left(\cos(ps)-\cos(pu)\right)d\lambda(u)\right)dp\right\vert>0.$$ Now, one can write for $s\in \mathbb{Z}_{h_n}\backslash \{0\}$ and $u\in A_s^{h}$,
\begin{equation*}
\cos(sp)-\cos(pu)=\cos(pu)(\cos((s-u)p)-1)-\sin(pu)(\sin((s-u)p)-(s-u)p)-\sin(pu)(s-u)p
\end{equation*}
and via RD get rid of the first two terms (i.e. they contribute to $B$ rather than $A$ in Remark~\ref{remark:RD}). It follows that it is sufficient to observe: 
$$\limsup_{n\to\infty}\frac{1}{h_n}\left\vert \int_{-K(h_n)}^{K(h_n)}\exp\left\{\sum_{k=1}^\infty(\cos(px_k)-1)w_k\right\}\left(\sum_{k=1}^nw_k\sin(px_k)\right)h_npdp\right\vert>0.$$
Finally, via DCT and evenness of the integrand, we need only have:
$$\int_{0}^\infty\left(\sum_{k=1}^\infty w_k\sin(px_k)\right)p\exp\left\{\sum_{k=1}^\infty(\cos(px_k)-1)w_k\right\}dp\ne 0.$$ One can in fact check that the integrand is strictly positive, as Lemma~\ref{lemma:somefunction_positive} shows, and thus the proof is complete. 
\item $\kappa(0)=\infty$. The example here works for the same reasons as it did in \eqref{convergence:proof:infinite_variation} of the proof of Proposition~\ref{propositon:convergence_kernels_diffusion_positive} (but here benefiting explicitly also from $\mu^h=0$). We only remark that Orey's condition is of course fulfilled with $\epsilon=1$, by checking it on the decreasing sequence $(h_n)_{n\geq 1}$.
\end{enumerate}
\end{proof}

\begin{lemma}\label{lemma:somefunction_positive}
Let $\psi(p):=\sum_{k=1}^\infty 3^{k/2}\sin(\frac{3}{2}p/3^k)$. Then $\psi$ is strictly positive on $(0,\infty)$.  
\end{lemma}
\begin{proof}
We observe that, (i) $\psi\vert_{(0,\frac{\pi}{2}]}>0$ and (ii) for $p\in (\pi/2,3\pi/2]$ we have: $\psi(p)>\sqrt{3}/(\sqrt{3}-1)=:A_0$. Indeed, (i) is trivial since for $p\in (0,\pi/2]$, $\psi(p)$ is defined as a sum of strictly positive terms. We verify (ii) by observing that (ii.1) $\psi(\pi/2)>A_0$ and (ii.2) $\psi$ is nondecreasing on $[\pi/2,3\pi/2]$. Both these claims are tedious but elementary to verify by hand. Indeed, with respect to (ii.1), summing three terms of the series defining $\psi(\pi/2)$ is sufficient. Specifically we have $\psi(\pi/2)> \sqrt{3}\sin(\pi/4)+3\sin(\pi/12)+3\sqrt{3}\sin(\pi/36)$ and we estimate $\sin(\pi/36)\geq \frac{\pi}{36}\sin(\pi/3)/(\pi/3)$. With respect to (ii.2) we may differentiate under the summation sign, and then $\psi'(p)\geq \frac{\sqrt{3}}{2}\cos(3\pi/4)+\frac{1}{2}\cos(\pi/4)+\frac{\sqrt{3}}{6}\cos(\pi/12)$. The final details of the calculations are left to the reader. 

Finally, we claim that if for some $B>0$ we have $\psi\vert_{(0,B]}>0$ and $\psi\vert_{(B,3B]}>A_0$, then $\psi\vert_{(0,3B]}>0$ and $\psi\vert_{(3B,9B]}>A_0$, and hence the assertion of the lemma will follow at once (by applying the latter first to $B=\pi/2$, then $B=3\pi/2$ and so on). So let $3p\in (3B,9B]$, i.e. $p\in (B,3B]$. Then $\psi(3p)=\sqrt{3}(\sin(3p/2)+\psi(p))>\sqrt{3}(-1+A_0)=A_0$, as required. 
\end{proof}

\section{Convergence of expectations and algorithm}\label{section:expectations:numerics}

\subsection{Convergence of expectations}\label{subsection:convergence_of_expectations}
For the sake of generality we state the results in the multivariate setting, but only do so, when this is not too burdensome on the brevity of exposition. For $d=1$, either the multivariate or the univariate schemes may be considered.

Let $f:\mathbb{R}^d\to \mathbb{R}$ be bounded Borel measurable and define for $t\geq 0$ and $h\in (0,h_\star)$: $p_t:=p_{0,t}$ and $P^h_t:=P^h_{0,t}$, whereas for $x\in \Zh^d$, we let $p_t(x):=p_t(0,x)$ and $P_t^h(x)=P_t^h(0,x)$ (assuming the continuous densities exist). Note that for $t\geq 0$, and then for $x\in \mathbb{R}^d$, 
\begin{equation}\label{eq:expectation:original}
\EE^x[f\circ X_t]=\int_\mathbb{R^d}f(y)p_t(x,y)dy,
\end{equation}
whereas for $x\in \Zh^d$ and $h\in (0,h_\star)$: 
\begin{equation}\label{eq:expectation:approximation}
\EE^x[f\circ X^h_t]=\sum_{y\in \Zh^d}f(y)P^h_t(x,y).
\end{equation}
Moreover, if $f$ is continuous, we know that, as $h\downarrow 0$, $\EE^x[f\circ X^h_t]\to \EE^x[f\circ X_t]$, since $X^h_t\to X_t$ in distribution. Next, under additional assumptions on the function $f$, we are able to establish the rate of this convergence and how it relates to the convergence rate of the transition kernels, to wit:

\begin{proposition}\label{proposition:expectations}
Assume \eqref{eq:assumption:multivariate:two} of Assumption~\ref{assumption:multivariate}. Let $h_0\in (0,\infty)$, $g:(0,h_0)\to (0,\infty)$ and $t>0$ be such that $\Delta_t=O(g)$. Suppose furthermore that the following two conditions on $f$ are satisfied:
\begin{enumerate}[(i)]
\item \label{expectation:i} $f$ is (piecewise\footnote{In the sense that there exists some natural $n$, and then disjoint open intervals $(I_i)_{i=1}^n$, whose union is cofinite in $\mathbb{R}$, and such that $f\vert_{I_i}$ is Lipschitz for each $i\in \{1,\ldots,n\}$.}, if $d=1$) Lipschitz continuous.
\item \label{expectation:ii} $\sup_{h\in (0,h_0)} h^d\sum_{x\in\Zh^d}\vert f(x)\vert<\infty$. 
\end{enumerate}
Then: $$\sup_{x\in \Zh^d}\vert \EE^x[f\circ X_t]-\EE^x[f\circ X^h_t]\vert=O(h\lor g(h)).$$
\end{proposition}

\begin{remark}
\begin{enumerate}
\item Condition \ref{expectation:ii} is fulfilled in the univariate case $d=1$, if, e.g.: $f\in L^1(\mathbb{R})$, w.r.t. Lebesgue measure, $f$ is locally bounded and for some $K\in [0,\infty)$, $\vert f\vert\vert_{(-\infty,-K]}$ (restriction of $\vert f\vert$ to $(-\infty,-K ]$) is nondecreasing, whereas $\vert f\vert\vert_{[K,\infty)}$ is nonincreasing. 
\item The rate of convergence of the expectations is thus got by combining the above proposition with the findings of Theorems~\ref{theorem:convergencerates} and~\ref{theorem:convergence:multivariate}. 
\end{enumerate}
\end{remark}

\begin{proof}
Decomposing the difference $\EE^x[f\circ X_t]-\EE^x[f\circ X^h_t]$ via \eqref{eq:expectation:original} and \eqref{eq:expectation:approximation}, we have: 
\begin{eqnarray}
\label{expectations:estimate:one}\EE^x[f\circ X_t]-\EE^x[f\circ X^h_t]&=&\sum_{y\in \Zh^d}\int_{A_y^h}\left(f(z)-f(y)\right)p_t(x,z)dz+\\
\label{expectations:estimate:two}&+&\sum_{y\in\Zh^d}\int_{A_y^h}f(y)\left(p_t(x,z)-p_t(x,y)\right) dz+\\
\label{expectations:estimate:three}&+&\sum_{y\in \Zh^d}f(y)h^d\left[p_t(x,y)-\frac{1}{h^d}P^h_t(x,y)\right].
\end{eqnarray}
Now, \eqref{expectations:estimate:three} is of order $O(g(h))$, by condition \ref{expectation:ii} and since $\Delta_t=O(g)$. Further, \eqref{expectations:estimate:one} is of order $O(h)$ on account of condition \ref{expectation:i}, and since $\int p_t(x,z)dz=1$ for any $x\in \mathbb{R}^d$ (to see piecewise Lipschitzianity is sufficient in dimension one ($d=1$), simply observe $\sup_{\{x,y\}\subset \mathbb{R}}p_t(x,y)$ is finite, as follows immediately from the integral representation of $p_t$). Finally, note that $p_t(x,\cdot)$ is also Lipschitz continuous (uniformly in $x\in \mathbb{R}^d$), as follows again at once from the integral representation of the transition densities. Thus, \eqref{expectations:estimate:two} is also of order $O(h)$, where again we benefit from condition \ref{expectation:ii} on the function $f$.
\end{proof}
In order to be able to relax condition \ref{expectation:ii} of Proposition~\ref{proposition:expectations}, we first establish the following Proposition~\ref{proposition:gmoments}, which concerns finiteness of moments of $X_t$. 

In preparation thereof, recall the definition of submultiplicativity of a function $g:\mathbb{R}^d\to [0,\infty)$:
\begin{equation}\label{eq:submultiplicativity}
g\text{ is submultiplicative}\Leftrightarrow \exists a\in (0,\infty)\text{ such that }g(x+y)\leq a g(x)g(y)\text{, whenever }\{x,y\}\subset \mathbb{R}^d
\end{equation}
and we refer to \cite[p. 159, Proposition 25.4]{sato} for examples of such functions. Any submultiplicative locally bounded function $g$ is necessarily bounded in exponential growth \cite[p. 160, Lemma 25.5]{sato}, to wit: 
\begin{equation}\label{eq:subexp}
\exists \{b,c\}\subset (0,\infty)\text{ such that }g(x)\leq be^{c\vert x\vert}\text{ for }x\in \mathbb{R}^d.
\end{equation}

\begin{proposition}\label{proposition:gmoments}
Let $g:\mathbb{R}^d\to [0,\infty)$ be measurable, submultiplicative and locally bounded, and suppose $\int_{\mathbb{R}^d\backslash [-1,1]^d} g d \Measure<\infty$. Then for any $t>0$, $\EE[g\circ X_t]<\infty$ and, moreover, there is an $h_0\in (0,h_\star)$ such that $$\sup_{h\in (0,h_0)} \EE[g\circ X^h_t]<\infty.$$ Conversely, if $\int_{\mathbb{R}^d\backslash [-1,1]^d} g d \Measure=\infty$, then for all $t>0$, $\EE[g\circ X_t]=\infty$.
\end{proposition}
\begin{proof}
The argument follows the exposition given in \cite[pp. 159-162]{sato}, modifying the latter to the extent that uniform boundedness over $h\in (0,h_0)$ may be got. In particular, we refer to \cite[p. 159, Theorem 25.3]{sato} for the claim that $\EE[g\circ X_t]<\infty$, if and only if $\int_{\mathbb{R}^d\backslash [-1,1]^d} g d \Measure<\infty$. We take $\{a,b,c\}\subset (0,\infty)$ satisfying \eqref{eq:submultiplicativity} and \eqref{eq:subexp} above. Recall also that $\lambda^h$ is the L\'evy measure of the process $X^h$, $h\in (0,h_\star)$. 

Now, decompose $X=X^1+X^2$ and $X^h=X^{h1}+X^{h2}$, $h\in (0,h_\star)$ as independent sums, where $X^1$ is compound Poisson, L\'evy measure $\lambda_1:=\mathbbm{1}_{\mathbb{R}^d\backslash [-1,1]^d}\cdot \Measure$, and $X^{h1}$ are also compound Poisson, L\'evy measures $\lambda_1^h:=\mathbbm{1}_{\mathbb{R}^d\backslash [-1,1]^d}\cdot \measure^h$, $h\in (0,h_\star)$. Consequently $X^2$ is a L\'evy process with characteristic triplet $(\Sigma,\Indd\cdot \Measure,\Drift)_{\tilde{c}}$ and $X^{h2}$ are compound Poisson, L\'evy measures $\Indd\cdot \Measure^h$, $h\in (0,h_\star)$. Moreover, for $h\in (0,h_\star)$, by submultiplicativity and independence: 
\begin{equation*}\label{equation:gmoments_estimate}
\EE[g\circ X^h_t]=\EE[g\circ (X^{h1}_t+X^{h2}_t)]\leq a\EE[g\circ X^{h1}_t]\EE[g\circ X^{h2}_t].
\end{equation*}

We first estimate $\EE[g\circ X^{h1}_t]$. Let $(J_n)_{n\geq 1}$ (resp. $N_t$) be the sequence of jumps (resp. number of jumps by time $t$) associated to (resp. of) the compound Poisson process $X^{h1}$. Then $X^{h1}_t=\sum_{j=1}^{N_t}J_j$ and so by submultiplicativity:
\begin{eqnarray*}
\EE[g\circ X^{h1}_t]&\leq&\EE\left[g(0)\mathbbm{1}_{\{N_t=0\}}+a^{N_t-1}\prod_{j=1}^{N_t}g(J_j)\mathbbm{1}_{\{N_t>0\}}\right]\\
&=&g(0)e^{-t\Measure^{h}_1(\mathbb{R}^d)}+\sum_{n=1}^\infty\frac{t^{n}a^{n-1}}{n!}e^{-t\Measure^{h}_1(\mathbb{R}^d)}\left(\int gd\Measure^h_1\right)^n.
\end{eqnarray*}
We also have for all $h\in (0,1\land h_\star)$:
\begin{eqnarray*}
\int gd\Measure^{h}_1&=&\sum_{s\in \Zh^d\backslash [-1,1]^d}\int_{A_s^h}g(s)d\measure=\sum_{s\in \Zh^d\backslash [-1,1]^d}\int_{A_s^h}g(u+(s-u))d\measure(u)\\
& \leq&a\left(\sup_{k\in A^0_h}g(k)\right)\sum_{s\in \Zh^d\backslash [-1,1]^d}\int_{A_s^h}gd\measure, \text{by submultiplicativity}\\
&\leq& a\left(\sup_{k\in A^0_1}g(k)\right)\int_{\mathbb{R}^d\backslash [-1/2,1/2]^d}gd\lambda.
\end{eqnarray*}
Now, since $g$ is locally bounded, $\Measure$ is finite outside neighborhoods of $0$, and since by assumption $\int_{\mathbb{R}^d\backslash [-1,1]^d} g d\Measure<\infty$, we obtain: $\sup_{h\in (0,1\land h_\star)}\EE[g\circ X^{h1}_t]<\infty$.

Second, we consider $\EE[g\circ X^{h2}_t]$. First, by boundedness in exponential growth and the triangle inequality: 
\begin{equation*}
\EE[g\circ X^{h2}_t]\leq b\EE[e^{c\vert X^{h2}_t\vert}]\leq b\EE[e^{c\sum_{j=1}^d\vert X^{h2}_{tj}\vert}]=b\EE\left[\prod_{j=1}^d e^{c\vert X^{h2}_{tj}\vert}\right].
\end{equation*}
It is further seen by a repeated application of the Cauchy-Schwartz inequality that it will be sufficient to show, for each $j\in \{1,\ldots,d\}$, that for some $h_0\in (0,h_\star]$: $$\sup_{h\in (0,h_0)}\EE\left[e^{2^{d-1}c\vert X^{h2}_{tj}\vert}\right]<\infty.$$ Here $X_t^{h2}=(X^{h2}_{t1},\ldots,X^{h2}_{td})$ and likewise for $X^2_t$. Fix $j\in \{1,\ldots,d\}$. 

The characteristic exponent of $X^{h2}_j$, denoted $\Psi^{h}_2$, extends to an entire function on $\mathbb{C}$. Likewise for the characteristic exponent of $X^2_j$, denoted $\Psi_2$ \cite[p. 160, Lemma 25.6]{sato}. Moreover, since, by expansion into power series, one has, locally uniformly in $\beta\in \mathbb{C}$, as $h\downarrow 0$:

\begin{itemize}
\item $\frac{e^{\beta h}+e^{-\beta h}-2}{2h^2}\to \frac{1}{2}\beta^2$;
\item $\frac{e^{\beta h}-e^{-\beta h}}{2h}\to \beta$;
\item $\frac{e^{\beta h}-1}{h}\to \beta$ and $\frac{1-e^{-\beta h}}{h}\to\beta$;
\end{itemize}
since furthermore: 
\begin{itemize}
\item $\left((\beta,u)\mapsto \frac{e^{\beta u}-\beta u-1}{u^2}\right):\mathbb{R}\backslash \{0\}\times \mathbb{C}\to \mathbb{C}$ is bounded on bounded subsets of its domain;
\end{itemize}
and since finally by the complex Mean Value Theorem  \cite[p. 859, Theorem 2.2]{evard}: 
\begin{itemize}
\item as applied to the function $(x\mapsto e^{\beta x}): \mathbb{C}\to\mathbb{C}$; $\vert e^{\beta x}-e^{\beta y}\vert\leq \vert x-y\vert\vert \beta\vert(\vert e^{\beta z_1}\vert+\vert e^{\beta z_2} \vert)$ for some $\{z_1,z_2\}\subset \conv(\{x,y\})$, for all $\{x,y\}\subset \mathbb{R}$;
\item as applied to the function $(x\mapsto e^{\beta x}-\beta x):\mathbb{C}\to\mathbb{C}$; $\vert e^{\beta x}-\beta x-(e^{\beta y}-\beta y)\vert\leq \vert x-y\vert \vert \beta\vert \left(\vert e^{\beta z_1}-1\vert+\vert e^{\beta z_2}-1\vert\right)$ for some $\{z_1,z_2\}\in  \conv(\{x,y\})$, for all $\{x,y\}\subset \mathbb{R}$;
\end{itemize}
then the usual decomposition of the difference $\Psi^{h}_2-\Psi_2$ (see proof of Proposition~\ref{proposition:someestimates}) shows that $\Psi^{h}_2\to\Psi_2$ locally uniformly in $\mathbb{C}$ as $h\downarrow 0$. Next let $\phi^{h}_2$ and $\phi_2$ be the characteristic functions of $X^{h2}_{tj}$ and $X^2_{tj}$, respectively, $h\in (0,h_\star)$; themselves entire functions on $\mathbb{C}$. Using the estimate of Lemma~\ref{lemma:keylemma}, we then see, by way of corollary, that also $\phi^{h}_2\to \phi_2$ locally uniformly in $\mathbb{C}$ as $h\downarrow 0$. 

Now, since $\phi_2^h$ is an entire function, for $n\in \mathbb{N}\cup \{0\}$, $i^n\EE[(X^{h2}_{tj})^n]=(\phi^h_2)^{(n)}(0)$ and it is Cauchy's estimate \cite[p. 184, Lemma 10.5]{stewart} that, for a fixed $r>2^{d-1}c$, $\left\vert(\phi^h_2)^{(n)}(0)\right\vert\leq \frac{n!}{r^n}M^h$, where $M^h:=\sup_{\{z\in \mathbb{C}:\vert z \vert=r\}}\vert \phi^h_2\vert$. Observe also that for some $h_0\in (0,h_\star]$, $\sup_{h\in (0,h_0)}M^h<\infty$, since $\phi^{h}_2\to \phi_2$ locally uniformly as $h\downarrow 0$ and $\phi_2$ is continuous (hence locally bounded).

Further to this $\EE[\vert X^{h2}_{tj}\vert^{2k+1}]\leq 1+ \EE[\vert X^{h2}_{tj}\vert^{2k+2}]$ ($k\in\mathbb{N}\cup\{0\}$) and $\EE\left[e^{2^{d-1}c\vert X^{h2}_{tj}\vert}\right]=\sum_{n=0}^\infty\frac{1}{n!}\EE[\vert X^{h2}_{tj}\vert^n](c2^{d-1})^n.$ From this the desired conclusion finally follows. 

\end{proof}
The following result can now be established in dimension $d=1$: 

\begin{proposition}\label{propositon:expectations:growth}
Let $d=1$ and $t>0$. Let furthermore:
\begin{enumerate}[(i)]
\item\label{expectation:condition:one} $g:\mathbb{R}\to [0,\infty)$, measurable, satisfy $\EE[g\circ X_t]<\infty$, $g$ locally bounded, submultiplicative, $g\ne 0$.
\item\label{expectation:condition:two} $f:\mathbb{R}\to\mathbb{R}$, measurable, be locally bounded, $\int_\mathbb{R}\vert f\vert \in (0,\infty]$, $\vert f\vert$ ultimately monotone (i.e. $\vert f\vert\vert_{[K,\infty)}$ and $\vert f\vert\vert_{(-\infty,-K]}$ monotone for some $K\in [0,\infty)$), $\vert f\vert/\vert g\vert$ ultimately nonincreasing (i.e. $(\vert f\vert/\vert g \vert)\vert_{[K,\infty)}$ and $(-\vert f\vert/\vert g \vert)\vert_{(-\infty,-K]}$ nonincreasing for some $K\in [0,\infty)$), and with the following Lipschitz property holding for some $\{a,A\}\in (0,\infty)$: $f\vert_{[-A,A]}$ is piecewise Lipschitz, whereas $$\vert f(z)-f(y)\vert\leq a \vert z-y\vert (g(z)+g(y)),\text{ whenever } \{z,y\}\subset \mathbb{R}\backslash (-A,A).$$
\item $K:(0,\infty)\to [0,\infty)$, with $\lim_{0+}K=+\infty$.
\end{enumerate}
Then $\vert \EE[f\circ X_t]-\EE[f\circ X^h_t]\vert$ is of order: 
\begin{equation}\label{expectations:growth}
O\left(\left(\int_{[-K(h),K(h)]}\mkern-70mu\vert f(x)\vert dx\right)\left(h\lor \Delta_t(h)\right)+\left(\frac{\vert f\vert}{\vert g\vert}\lor \frac{\vert f\vert}{\vert g\vert}\circ (-\mathrm{id}_\mathbb{R})\right)\left(K(h)-3h/2\right)\right),
\end{equation}
where $\Delta_t(h)$ is defined in \eqref{eq:fundamental_notation}.
\end{proposition}

\begin{remark}\label{remark:expectations:growth}
\begin{enumerate}
\item In \eqref{expectations:growth} there is a balance of two terms, viz. the choice of the function $K$. Thus, the slower (resp. faster) that $K$ increases to $+\infty$ at $0+$, the better the convergence of the first (resp. second) term, provided $f\notin L^1(\mathbb{R})$ (resp. $\vert f\vert/\vert g\vert$ is ultimately converging to $0$, rather than it just being nonincreasing). In particular, when so, then the second term can be made to decay arbitrarily fast, whereas the first term will always have a convergence which is strictly worse than $h\lor \Delta_t(h)$. But this convergence can be made arbitrarily close to $h\lor \Delta_t(h)$ by choosing $K$ increasing all the slower (this since $f$ is locally bounded). In general the choice of $K$ would be guided by balancing the rate of decay of the two terms.
\item Since, in the interest of relative generality, (further properties of) $f$ and $\measure$ are not specified, thus also $g$ cannot be made explicit. Confronted with a specific $f$ and L\'evy process $X$, we should like to choose $g$ approaching infinity (at $\pm\infty$) as fast as possible, while still ensuring $\EE[g\circ X_t]<\infty$ (cf. Proposition~\ref{proposition:gmoments}). This makes, \emph{ceteris paribus}, the second term in \eqref{expectations:growth} decay as fast as possible. 
\item We exemplify this approach by considering two examples. Suppose for simplicity $\Delta_t(h)=O(h)$. 
\begin{enumerate}
\item\label{remark:expectations:growth:polynomial} Let first $\vert f\vert $ be bounded by $(x\mapsto A\vert x\vert^n)$ for some $A\in (0,\infty)$ and $n\in \mathbb{N}$, and assume that for some $m\in (n,\infty)$, the function $g=(x\mapsto \vert x\vert^m\lor 1)$ satisfies $\EE[g\circ X_t]<\infty$ (so that \ref{expectation:condition:one} holds). Suppose furthermore condition \ref{expectation:condition:two} is satisfied as well (as it is for, e.g., $f=(x\mapsto x^n)$). It is then clear that the first term of \eqref{expectations:growth} will behave as $\sim K(h)^{n+1}h$, and the second as $\sim K(h)^{-(m-n)}$, so we choose $K(h)\sim 1/h^{1/(1+m)}$ for a rate of convergence which is of order $O(h^{\frac{m-n}{m+1}})$. 
\item\label{remark:expectations:growth:exponential} Let now $\vert f\vert $ be bounded by $(x\mapsto Ae^{\alpha\vert x\vert})$ for some $\{A,\alpha\}\subset (0,\infty)$, and assume that for some $\beta\in (\alpha,\infty)$, the function $g=(x\mapsto e^{\beta\vert x\vert})$ indeed satisfies  $\EE[g\circ X_t]<\infty$ (so that \ref{expectation:condition:one} holds). Suppose furthermore condition \ref{expectation:condition:two} is satisfied as well (as it is for, e.g., $f=(x\mapsto (e^{\alpha x}-k)^+)$, where $k\in [0,\infty)$ --- use Lemma~\ref{lemma:keylemma}). It is then clear that the first term of \eqref{expectations:growth} will behave as $\sim e^{\alpha K(h)}h$, and the second as $\sim e^{-(\beta-\alpha)K(h)}$, so we choose, up to a bounded additive function of $h$, $K(h)=\log (1/h^{1/\beta})$ for a rate of convergence which is of order $O(h^{1-\frac{\alpha}{\beta}})$. 
\end{enumerate}
\end{enumerate}
\end{remark}
\begin{proof}(Of Proposition~\ref{propositon:expectations:growth}.)
This is a simple matter of estimation; for all sufficiently small $h>0$:
\begin{eqnarray}\nonumber
&&\vert \EE[f\circ X_t]-\EE[f\circ X^h_t]\vert=\left\vert \int_\mathbb{R}f(z)p_t(z)dz-\sum_{y\in\Zh}f(y)P^h_t(y)\right\vert\\\nonumber
&\leq&\left\vert \sum_{y\in [-K(h),K(h)]\cap \Zh}\left(\int_{A_y^h}f(z)p_t(z)dz-f(y)P_t^h(y)\right)\right\vert+\sum_{y\in \Zh\backslash [-K(h),K(h)]}\vert f(y)\vert P_t^h(y)+\\\nonumber
&& \int_{\mathbb{R}\backslash [-(K(h)-h/2),K(h)-h/2]}\left\vert f(z)\right\vert p_t(z)dz\\\nonumber
&\leq& \underbrace{\left\vert \sum_{y\in \Zh\cap [-K(h),K(h)]}\int_{A_y^h}\left(f(z)-f(y)\right)p_t(z)dz\right\vert}_{\mytagg{exp_conv:1}}+\underbrace{\left\vert \sum_{y\in\Zh\cap [-K(h),K(h)]}\int_{A_y^h}f(y)\left(p_t(z)-p_t(y)\right) dz\right\vert}_{\mytagg{exp_conv:2}}+\\\nonumber
&&\underbrace{\left\vert\sum_{y\in \Zh\cap [-K(h),K(h)]}f(y)h\left[p_t(y)-\frac{1}{h}P^h_t(y)\right]\right\vert}_{\mytagg{exp_conv:3}}+\\\nonumber
&&\underbrace{\left(\frac{\vert f\vert}{\vert g\vert}\lor \frac{\vert f\vert}{\vert g\vert}\circ (-\mathrm{id}_\mathbb{R})\right)\left(K(h) \right)\EE[g\circ X_t^h]}_{\mytagg{exp_conv:4}}+\underbrace{\left(\frac{\vert f\vert}{\vert g\vert}\lor \frac{\vert f\vert}{\vert g\vert}\circ (-\mathrm{id}_\mathbb{R})\right)\left(K(h)-h/2 \right)\EE[g\circ X_t]}_{\mytagg{exp_conv:5}}.
\end{eqnarray}
Thanks to Proposition~\ref{proposition:gmoments}, and the fact that $\vert f\vert/\vert g\vert$ is ultimately nonincreasing, \ref{exp_conv:4} \& \ref{exp_conv:5} are bounded (modulo a multiplicative constant) by $\frac{\vert f\vert}{\vert g\vert}(K(h)-h/2)\lor \frac{\vert f\vert}{\vert g\vert}(-(K(h)-h/2))$. From the Lipschitz property of $f$, submultiplicativity and local boundedness of $g$, and the fact that $\EE[g\circ X_t]<\infty$, we obtain \ref{exp_conv:1} is of order $O(h)$. By the local boundedness and eventual monotonicity of $\vert f\vert$, the Lipschitz nature of $p_t$ and the fact that $\int\vert f\vert>0$, \ref{exp_conv:2} is bounded (modulo a multiplicative constant) by $h\int_{[-(K(h)+h),K(h)+h]}\vert f\vert$. Finally, a similar remark pertains to \ref{exp_conv:3}, but with $\Delta_t(h)$ in place of $h$. Combining these, using once again $\int \vert f \vert>0$, yields the desired result, since we may finally replace $K(h)$ by $(K(h)-h)\lor 0$.
\end{proof}

\subsection{Algorithm}\label{subsection:algorithm}
From a numerical perspective we must ultimately consider the processes $X^h$ on a finite state space, which we take to be $S^h_M:=\{x\in \Zh^d:\vert x\vert\leq M\}$ ($M>0$, $h\in (0,h_\star)$). We let $\hat{Q}^h$ denote the sub-Markov generator got from $Q^h$ by restriction to $S^h_M$, and we let $\hat{X}^h$ be the corresponding Markov chain got by killing $X^h$ at the time $T^h_M:=\inf\{t\geq 0:\vert X_t^h\vert>M\}$, sending it to the coffin state $\partial$ thereafter.

Then the basis for the numerical evaluations is the observation that for a (finite state space) Markov chain $Y$ with generator matrix $Q$, the probability $\PP^y(Y_t=z)$ (resp. the expectation $\EE^y[f\circ Y]$, when defined) is given by $(e^{tQ})_{yz}$ (resp. $(e^{t Q}f)(y)$). With this in mind we propose the: 
\begin{center}
\noindent\fbox{
    \parbox{0.8\textwidth}{
\begin{center}\textbf{Sketch algorithm}\end{center}
\begin{enumerate}[(i)]
\item Choose $\{h,M\}\subset (0,\infty)$.
\item Calculate, for the truncated sub-Markov generator $\hat{Q}^h$, the matrix exponential $\exp\{t \hat{Q}^h\}$ or action $\exp\{t\hat{Q}^h\}f$ thereof (where $f$ is a suitable vector). 
\item Adjust truncation parameter $M$, if needed, and discretization parameter $h$, until sufficient precision has been established. 
\end{enumerate}
    }
}
\end{center}
Two questions now deserve attention: (1) what is the truncation error and (2) what is the expected cost of this algorithm. We address both in turn. 

First, with a view to the localization/truncation error, we shall find use of the following:

\begin{proposition}\label{propositiion:bounded_supremum}
Let $g: [0,\infty)\to [0,\infty)$ be nondecreasing, continuous and submultiplicative, with $\lim_{+\infty}g=+\infty$. Let $t>0$ and denote by:
$$X^\star_t=\sup_{s\in [0,t]}\vert X_s\vert,\quad X^{h\star}_t=\sup_{s\in [0,t]}\vert X^h_s\vert,$$ the running suprema of $\vert X\vert$ and of $\vert X^h\vert$, $h\in (0,h_\star)$, respectively. Suppose furthermore $\EE[g\circ \vert X_t\vert]<\infty$. Then $\EE [g\circ X^\star_t]<\infty$ and, moreover, there is some $h_0\in (0,h_\star]$ such that $$\sup_{h\in (0,h_0)}\EE[g\circ X^ {h\star}_t]<\infty.$$ 
\end{proposition}
\begin{remark}\label{remark:bounded_supremum}
The function $g\circ \vert \cdot\vert:\mathbb{R}^d\to[0,\infty)$ is measurable, submultiplicative and locally bounded, so for a condition on the L\'evy measure equivalent to $\EE[g\circ X_t]<\infty$ see Proposition~\ref{proposition:gmoments}. 
\end{remark}
We prove Proposition~\ref{propositiion:bounded_supremum} below, but first let us show its relation to the truncation error. For a function $f:\Zh^d\to \mathbb{R}$, we extend its domain to $\Zh^d\cup \{\partial\}$, by stipulating that $f(\partial)=0$. The following (very crude) estimates may then be made:

\begin{corollary}\label{corollary:truncation}
Fix $t>0$. Assume the setting of Proposition~\ref{propositiion:bounded_supremum}. There is some $h_0\in (0,h_\star]$  and then $C:=\sup_{h\in (0,h_0)}\EE[g\circ X^ {h\star}_t]<\infty$, such that the following two claims hold:
\begin{enumerate}[(i)]
\item\label{corollary:one} For all $h\in (0,h_0)$: $$\sum_{x\in \Zh^d}\vert \PP(X^h_t=x)-\PP(\hat{X}^h_t=x)\vert\leq \PP(T^h_M< t)\leq C/g(M).$$
\item\label{corollary:two} Let $f:\Zh^d\to \mathbb{R}$ and suppose $\vert f\vert\leq \tilde{f}\circ \vert \cdot \vert$, with $\tilde{f}:[0,\infty)\to [0,\infty)$ nondecreasing and such that $\tilde{f}/g$ is (resp. ultimately) nonincreasing. Then for all (resp. sufficiently large) $M>0$ and $h\in (0,h_0)$: $$\vert \EE[f\circ X^h_t]-\EE[f\circ \hat{X}^h_t]\vert\leq C\left(\frac{\tilde{f}}{g}\right)(M).$$
\end{enumerate}
\end{corollary}

\begin{remark}\label{remark:truncation}
\begin{enumerate}
\item Ad \ref{corollary:one}. Note that $M$ may be taken fixed (i.e. independent of $h$) and chosen so as to satisfy a prescribed level of precision. In that case such a choice may be verified explicitly at least retrospectively: the sub-Markov generator $\hat{Q}^h$ gives rise to the sub-Markov transition matrix $\hat{P}^h_t:=e^{t\hat{Q}^h}$; its deficit (in the row corresponding to state $0$) is precisely the probability $\PP(T^h_M< t)$. 
\item\label{remark:truncation:two} Ad~\ref{corollary:two}.  But $M$ may also be made to depend on $h$, and then let to increase to $+\infty$ as $h\downarrow 0$, in which case it is natural to balance the rate of decay of $\vert \EE[f\circ X^h_t]-\EE[f\circ \hat{X}^h_t]\vert$ against that of $\vert \EE[f\circ X_t]-\EE[f\circ X^h_t]\vert$ (cf. Proposition~\ref{propositon:expectations:growth}). In particular, since $\EE[g\circ \vert X_t\vert]<\infty\Leftrightarrow \EE[g\circ X^\star_t]\Leftrightarrow \int_{\mathbb{R}^d\backslash [-1,1]^d} g\circ \vert\cdot \vert d\measure<\infty$ \cite[p. 159, Theorem 25.3 \& p. 166, Theorem 25.18]{sato}, this problem is essentially analogous to the one in Proposition~\ref{propositon:expectations:growth}. In particular, Remark~\ref{remark:expectations:growth} extends in a straightforward way to account for the truncation error, with $M$ in place of $K(h)-3h/2$.
\end{enumerate}
\end{remark}

\begin{proof}
\ref{corollary:one} follows from the estimate $\sum_{x\in \Zh^d}\vert \PP(X^h_t=x)-\PP(\hat{X}^h_t=x)\vert\leq \PP(T_M^h<t)=\PP(X^{h\star}_t> M)\leq \frac{\EE[g\circ X^{h\star}_t]}{g(M)}$, which is an application of Markov's inequality. When it comes to \ref{corollary:two}, we have for all (resp. sufficiently large) $M>0$: 
\begin{eqnarray*}
&&\vert \EE[f\circ X^h_t]-\EE[f\circ \hat{X}^h_t]\vert\leq\EE\left[\left(\vert f\vert\circ X^h_t\right)\mathbbm{1}(T^h_M<t)\right]\leq \EE\left[\left(\tilde{f}\circ \vert X^h_t\vert\right)\mathbbm{1}(T^h_M<t)\right]\\
&\leq&\EE\left[\left(\tilde{f}\circ X^{h\star}_t\right)\mathbbm{1}(T^h_M<t)\right]=\EE\left[\left(\left(\frac{\tilde{f}}{g}\right)\circ X^{h\star}_t \right)\left(g\circ X^{h\star}_t\right)\mathbbm{1}(X^{h\star}_t>M)\right]\\
&\leq&\left(\frac{\tilde{f}}{g}\right)(M)\EE[g\circ X^{h\star}_t],\\
\end{eqnarray*}
whence the desired conclusion follows.
\end{proof}

\begin{proof} (Of Proposition~\ref{propositiion:bounded_supremum}.)
We refer to \cite[p. 166, Theorem 25.18]{sato} for the proof that $\EE[g\circ X^\star_t]<\infty$. Next, by right continuity of the sample paths of $X$, we may choose $b>0$, such that $\PP(X^*_t\leq b/2)>0$ and we may also insist on $b/2$ being a continuity point of the distribution function of $X^\star_t$ (there being only denumerably many points of discontinuity thereof). Now, $X^h\to X$ as $h\downarrow 0$ w.r.t. the Skorokhod topology on the space of c\`adl\`ag paths. Moreover,  by \cite[p. 339, 2.4 Proposition]{jacod}, the mapping $\Phi:=(\alpha\mapsto \sup_{s\in [0,t]}\vert \alpha(s)\vert):\mathbb{D}([0,\infty),\mathbb{R}^d)\to\mathbb{R}$ is continuous at every point $\alpha$ in the space of c\`adl\`ag paths $\mathbb{D}([0,\infty),\mathbb{R}^d)$, which is continuous at $t$. In particular, $\Phi$ is continuous, a.s. w.r.t. the law of the process $X$ on the Skorokhod space \cite[p. 59, Theorem 11.1]{sato}. By the Portmanteau Theorem, it follows that there is some $h_0\in (0,h_\star]$ such that $\inf _{h\in (0,h_0)}\PP(X^{h\star}_t\leq b/2)>0$. 

Moreover, from the proof of  \cite[p. 166, Theorem 25.18]{sato}, by letting $\tilde{g}:[0,\infty)\to [0,\infty)$ be nondecreasing, continuous, vanishing at zero and agreeing with $g$ on restriction to $[1,\infty)$, we may then show for each $h\in (0,h_\star)$ that: 
$$\EE[\tilde{g}\circ (X^{h\star}_t-b); X^ { h_\star}_t>b]\leq \EE[\tilde{g}\circ \vert X_t^h\vert]/\PP(X^{h\star}_t\leq b/2).$$
Now, since $\EE[g\circ \vert X_t\vert]<\infty$, by Proposition~\ref{proposition:gmoments} (cf. Remark~\ref{remark:bounded_supremum}), there is some $h_0\in (0,h_\star]$ such that $\sup_{h\in (0,h_0)}\EE[g\circ \vert X_t^h\vert]<\infty$, and thus  $\sup_{h\in (0,h_0)}\EE[\tilde{g}\circ \vert X_t^h\vert]<\infty$. 

Combining the above, it follows that for some $h_0\in (0,h_\star]$, $\sup_{h\in (0,h_0)}\EE[\tilde{g}\circ (X^{h\star}_t-b); X^ { h_\star}_t>b]<\infty$ and thus $\sup_{h\in (0,h_0)}\EE[g\circ (X^{h\star}_t-b); X^ { h_\star}_t>b]<\infty$. Finally, an application of submultiplicativity of $g$ allows to conclude. 
\end{proof}

Having thus dealt with the truncation error, let us briefly discuss the cost of our algorithm. 

The latter is clearly governed by the calculation of the matrix exponential, or, resp., of its action on some vector. Indeed, if we consider as fixed the generator matrix $\hat{Q}^h$, and, in particular, its dimension $n\sim (M/h)^d$, then this may typically require $O(n^3)$ \cite{moler,higham_scaling}, resp. $O(n^2)$ \cite{higham_exp_to_vector}, floating point operations. Note, however, that this is a notional complexity analysis of the algorithm. A more detailed argument would ultimately have to specify precisely the particular method used to determine the (resp. action of a) matrix exponential, and, moreover, take into account how $\hat{Q}^h$ (and, possibly, the truncation parameter $M$, cf. Remark~\ref{remark:truncation}) behave as $h\downarrow 0$. Further analysis in this respect goes beyond the desired scope of this paper. 

We finish off by giving some numerical experiments in the univariate case. To compute the action of $\hat{Q}^h$ on a vector we use the MATLAB function \texttt{expmv.m} \cite{higham_exp_to_vector}, unless $\hat{Q}^h$ is sparse, in which case we use the MATLAB function \texttt{expv.m} from \cite{sidje}. 

We begin with transition densities. To shorten notation, fix the time $t=1$ and allow $p:=p_1(0,\cdot)$ and $p^h:=\frac{1}{h}\hat{P}^h_1(0,\cdot)$ ($\hat{P}^h$ being the analogue of $P^h$ for the process $\hat{X}^h$). Note that to evaluate the latter, it is sufficient to compute $(e^{\hat{Q}^h t})_{0\cdot}=e^{(\hat{Q}^h)'t}\mathbbm{1}_{\{0\}}$, where $(\hat{Q}^h)'$ denotes transposition. 

\begin{example}\label{example:BM_drift}
Consider first Brownian motion with drift, $\diffusion=1$, $\drift=1$, $\measure=0$ (scheme 1, $V=0$). We compare the density $p$ with the approximation $p^h$ ($h\in \{1/2^n:n\in \{0,1,2,3\}\}$) on the interval $[-1,1]$ (see Figure~\ref{fig:convergence_BM_drift}), choosing $M=5$. The vector of deficit probabilities $(\PP(T^{1/2^n}_M<t))_{n=0}^3$ corresponding to using this truncation was $(5.9\cdot 10^{-4},1.5\cdot 10^{-4},5.8\cdot 10^{-5},4.4\cdot 10^{-5})$. In this case the matrix $\hat{Q}^h$ is sparse.
\end{example}

\begin{figure}[!htb]
\includegraphics[width=0.8\textwidth]{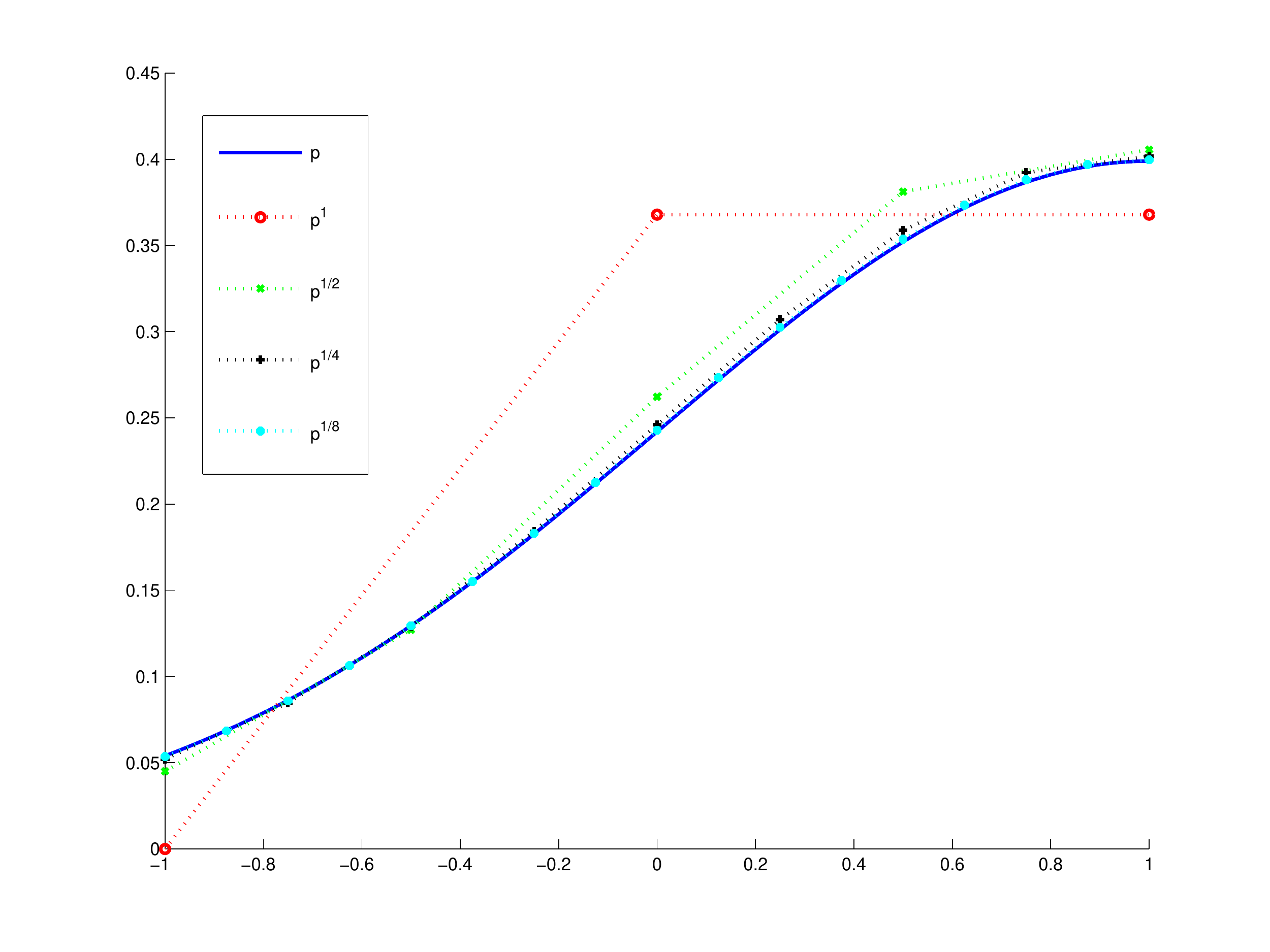}
\caption{Convergence of $p^h$ to $p$ (as $h\downarrow 0$) on the interval $[-1,1]$ for Brownian motion with drift ($\diffusion=\mu=1$, $\measure=0$, scheme 1, $V=0$). See Example~\ref{example:BM_drift} for details.}
\label{fig:convergence_BM_drift}
\end{figure}

\begin{example}\label{example:alpha_stable}
Consider now $\alpha$-stable L\'evy processes, $\diffusion=0$, $\drift=0$, $\measure(dx)=dx/\vert x\vert^{1+\alpha}$ (scheme 2, $V=1$). We compare the density $p$ with $p^h$ on the interval $[0,1]$ (see Figure~\ref{fig:convergence_alpha_stable}).  Computations are made for the vector of alphas given by $(\alpha_k)_{k=1}^4:=(1/2,1,4/3,5/3)$ with corresponding truncation parameters ($M_k)_{k=1}^4=(500,100,30,20)$ resulting in the deficit probabilities (uniformly over the $h$ considered) of $(\PP(T^{h}_{M_k}<t))_{k=1}^4=(1.7\cdot 10^{-1},2.0\cdot 10^{-2},(\text{from }1.7\text{ to }1.8)\cdot 10^{-2},(\text{from }0.94\text{ to }1.01)\cdot 10^{-2})$. The heavy tails of the L\'evy density necessitate a relatively high value of $M$. Nevertheless, excluding the case $\alpha=5/3$, a reduction of $M$ by a factor of $5$ resulted in an absolute change of the approximating densities, which was at most of the order of magnitude of the discretization error itself. Conversely, for $\alpha=1/2$, when the deficit probability is highest and appreciable, increasing $M$ by a factor of $2$, resulted in an absolute change of the calculated densities of the order $10^{-6}$ (uniformly over $h\in \{1,1/2,1/4\}$). Finally, note that $\alpha=1$ gives rise to the Cauchy distribution, whereas otherwise we use the MATLAB function \texttt{stblpdf.m} to get a benchmark density against which a comparison can be made. 
\end{example}

\begin{figure}[!htb]
\includegraphics[width=\textwidth]{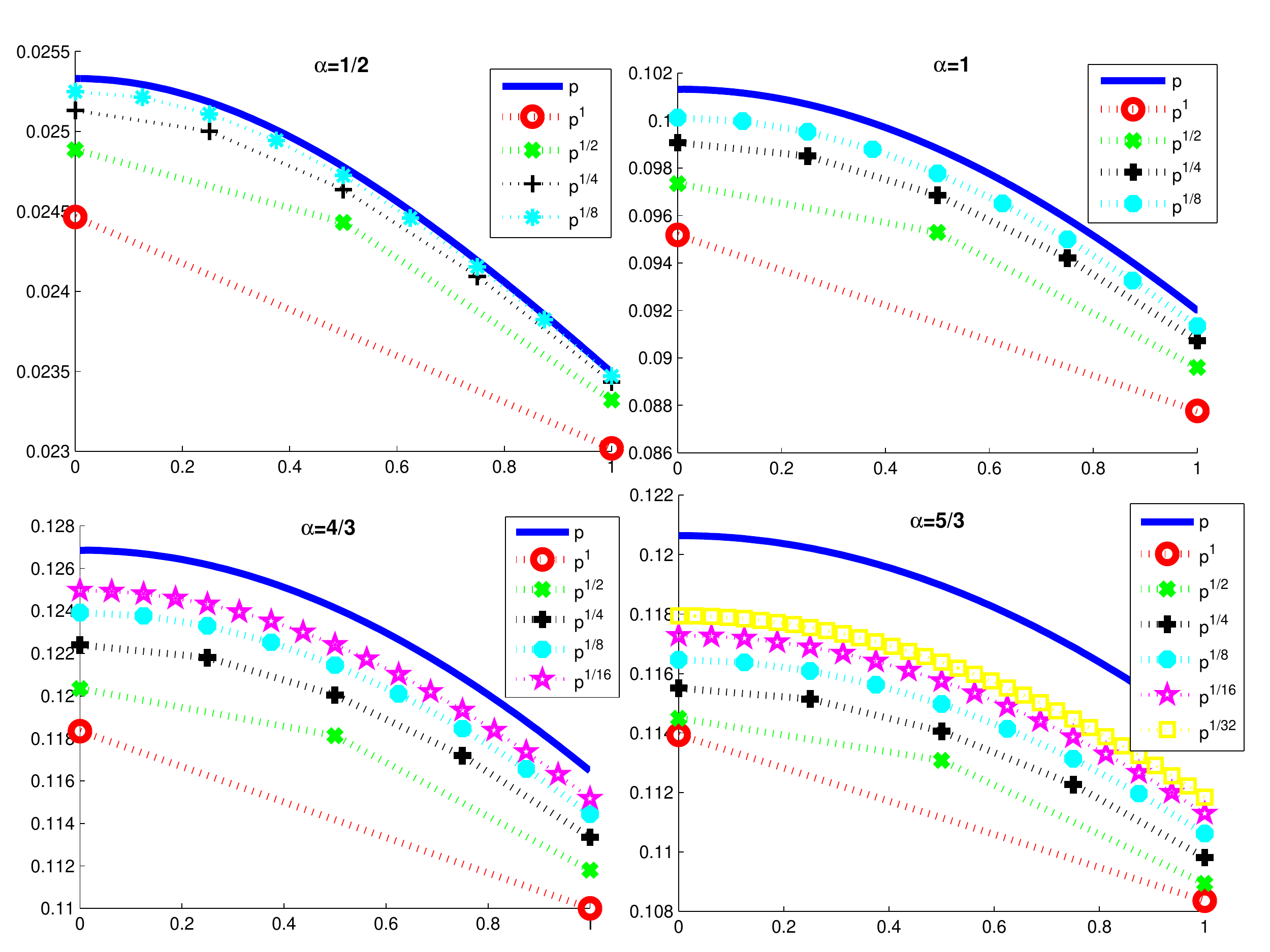}
  \caption{Convergence of $p^h$ to $p$ (as $h\downarrow 0$) on the interval $[0,1]$ for $\alpha$-stable L\'evy processes ($\diffusion=0$, $\drift=0$, $\measure(dx)=dx/\vert x\vert^{1+\alpha}$, scheme 2, $V=1$), $\alpha\in \{1/2,1,4/3,5/3\}$. See Example~\ref{example:alpha_stable} for details. Note that convergence becomes progressively worse as $\alpha\uparrow$, which is \emph{precisely consistent} with Figure~\ref{fig:comparison} and the theoretical order of convergence, this being $O(h^{(2-\alpha)\land 1})$ (up to a slowly varying factor $\log(1/h)$, when $\alpha=1$; and noting that Orey's condition is satisfied with $\epsilon=\alpha$). For example, when $\alpha=5/3$ each successive approximation should be closer to the limit by a factor of $\left(\frac{1}{2}\right)^ {1/3}\doteq 0.8$, as it is.}
  \label{fig:convergence_alpha_stable}
\end{figure}

\begin{example}\label{example:VG}
A particular VG model \cite{madan,mcc} has $\diffusion=0$, $\drift=0$, $\measure(dx)= \frac{e^{-\vert x\vert}}{\vert x\vert}\mathbbm{1}_{\mathbb{R}\backslash \{0\}}(x)dx$ (scheme 2, $V=1$). Again we compare $p$ with $p^h$  ($h\in \{1/2^n:n\in \{0,1,2,3\}\}$) on the interval $[0,1]$ (see Figure~\ref{fig:convergence_VG}), choosing $M=5$. The vector of deficit probabilities $(\PP(T^{1/2^n}_M<t))_{n=0}^3$ corresponding to using this truncation was $(5.2\cdot 10^{-3},6.4\cdot 10^{-3},7.2\cdot 10^{-3},7.6\cdot 10^{-3})$. The density $p$ is given explicitly by $(x\mapsto e^{-\vert x\vert}/2)$. 
\end{example}

\begin{figure}[!htb]
\includegraphics[width=0.9\textwidth]{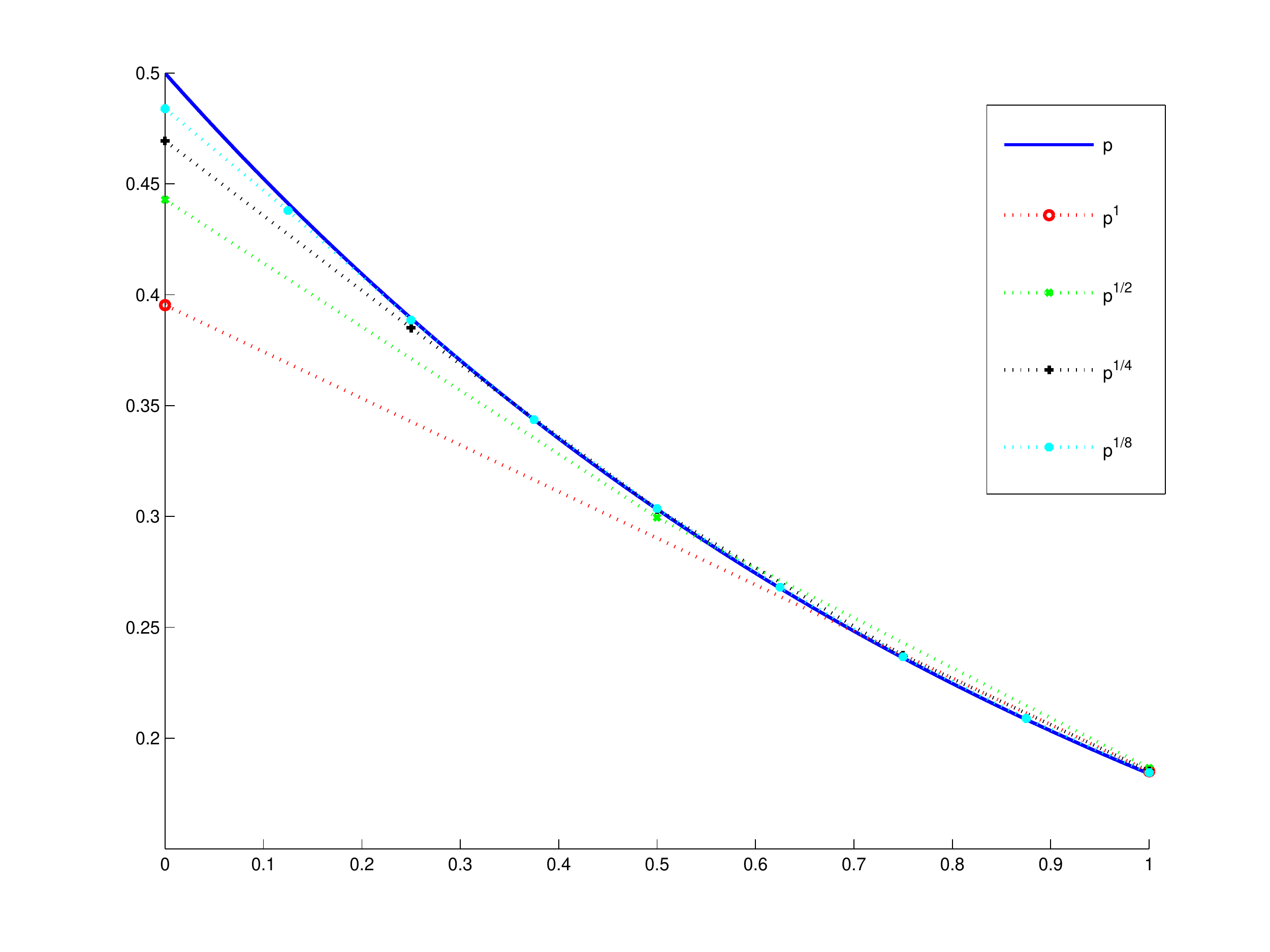}
\caption{Convergence of $p^h$ to $p$ (as $h\downarrow 0$) on the interval $[0,1]$ for the VG model ($\diffusion=0$, $\drift=0$, $\measure(dx)= \frac{e^{-\vert x\vert}}{\vert x\vert}\mathbbm{1}_{\mathbb{R}\backslash \{0\}}(x)dx$, scheme 2, $V=1$). Note that in this case Orey's condition fails, but (at least as evidenced numerically) linear convergence does not.  See Example~\ref{example:VG} for details.}
\label{fig:convergence_VG}
\end{figure}

Finally, to illustrate convergence of expectations, we consider a particular option pricing problem. 

\begin{example}\label{example:options}
Suppose that, under the pricing measure, the stock price process $S=(S_t)_{t\geq 0}$ is given by $S_t=S_0e^{rt+X_t}$, $t\geq 0$, where $S_0$ is the initial price, $r$ is the interest rate, and $X$ is a tempered stable process with L\'evy measure given by: $$\measure(dx)=c\left(\frac{e^{-\lambda_+x}}{x^{1+\alpha}}\mathbbm{1}_{(0,\infty)}(x)+\frac{e^{-\lambda_-\vert x\vert}}{\vert x\vert^{1+\alpha}}\mathbbm{1}_{(-\infty,0)}(x)\right)dx.$$ To satisfy the martingale condition, we must have $\EE[e^{X_t}]\equiv 1$, which in turn uniquely determines the drift $\mu$ (we have, of course, $\diffusion=0$). The price of the European put option with maturity $T$ and strike $K$ at time zero is then given by: $$P(T,K)=e^{-rT}\EE[(K-S_T)^+].$$ We choose the same value for the parameters as \cite{poirot}, namely $S_0=100$, $r=4\%$, $\alpha=1/2$, $c=1/2$, $\lambda_+=3.5$, $\lambda_-=2$ and $T=0.25$, so that we may quote the reference values $P(T,K)$ from there. 

Now, in the present case, $X$ is a process of finite variation, i.e. $\kappa(0)<\infty$, hence convergence of densities is of order $O(h)$, since Orey's condition holds with $\epsilon=1/2$ (scheme 2, $V=1$). Moreover, $\mathbbm{1}_{\mathbb{R}\backslash [-1,1]}\cdot \measure$ integrates $(x\mapsto e^{2\vert x\vert})$, whereas the function $(x\mapsto (K-e^{rt+x})^+)$ is bounded. Pursuant to \eqref{remark:truncation:two} of Remark~\ref{remark:truncation} we thus choose $M=M(h):=\left(\frac{1}{2}\log(1/h)\right)\lor 1$, which by Corollary~\ref{corollary:truncation} and Proposition~\ref{propositon:expectations:growth} (with $K(h)=M(h)$) (cf. also \eqref{remark:expectations:growth:exponential} of Remark~\ref{remark:expectations:growth}) ensures that: $$\vert \hat{P}^h(T,K)-P(T,K)\vert=O(h\log(1/h)),$$ where $\hat{P}^h(T,K):=e^{-rT}\EE[(K-S_0e^{rT+\hat{X}^h_T})^+]$. Table~\ref{table:options} summarizes this convergence on the decreasing sequence $h_n:=1/2^n$, $n\geq 1$. 

In particular, we wish to emphasize that the computations were all (reasonably) fast. For example, to compute the vector  $(\hat{P}^{h_n}(T,K))_{n=1}^9$ with $K=80$, the times (in seconds; entry-by-entry) $( 0.0106  ,  0.0038 ,   0.0044 ,   0.0078  ,  0.0457 ,   0.0367 ,   0.0925 ,   0.4504   , 2.4219)$ were required on an Intel 2.53 GHz processor (times obtained using MATLAB's \texttt{tic-toc} facility). This is much better than, e.g., the Monte Carlo method of \cite{poirot} and comparable with the finite difference method of \cite{contvoltchkova} (VG2 model in \cite[p. 1617, Section 7]{contvoltchkova}).
\end{example}

\begin{table}
\caption{Convergence of the put option price for a CGMY model (scheme 2, $V=1$). See Example~\ref{example:options} for details.}\label{table:options}
\begin{tabular}{|c|c|c|c|c|c|c|c|c|c| }\hline
$K\rightarrow $ & $80$ & $85$ & $90$ & $95$ & $100$ & $105$ & $110$ & $115$ & $120$\\\hline
$P(T,K)\rightarrow$ & 1.7444 & 2.3926& 3.2835& 4.5366& 6.3711& 9.1430& 12.7631& 16.8429& 21.1855\\\hline 
$n$ & \multicolumn{9}{|c|}{$\hat{P}^{h_n}(T,K)-P(T,K)$}\\\hline
1 &     0.6411  &  0.5422  &   0.2006  & -0.5033 &  -1.7885  & -0.8227  &   0.0970  &  0.5570  &   0.7542\\\hline
2 &   -0.1089  & 0.2816 &   0.4295  &  0.2151 &  -0.5806 &   0.0975  &  0.5341 &   0.5109  &   0.2250\\\hline
3 &    -0.2271 &  -0.1596 &   -0.1928  &  0.0920   &-0.2046  &  0.1405  &  0.0348  & -0.4356  & -0.3937\\\hline
4 &    -0.0904  & -0.0753 &  -0.0517  & -0.0442  &  0.0652  &   0.1487 &   0.0057  & -0.1511  & -0.1838\\\hline
5 &    -0.0411 &  -0.0338 &  -0.0193 &  -0.0053 &   0.0679 &   0.0569 &  -0.0073  & -0.0616  & -0.0833\\\hline
6 &   -0.0184  & -0.0163  & -0.0081  &  0.0022 &   0.0347 &   0.0314 &  -0.0033 &  -0.0244  & -0.0384\\\hline
7 &    -0.0079  &   -0.0069  &   -0.0040  &    0.0019  &    0.0152   &   0.0109  &   -0.0034  &   -0.0108   &  -0.0164\\\hline
8 &    -0.0034   & -0.0029  &  -0.0016 &    0.0011 &    0.0072 &    0.0053 &   -0.0012 &   -0.0048 &   -0.0070\\\hline
9 &    -0.0014  & -0.0012  & -0.0007  &  0.0006  &  0.0033  &  0.0026 &  -0.0004  & -0.0020 &  -0.0030\\\hline

\end{tabular}

\end{table}

In conclusion, the above numerical experiments serve to indicate that our method behaves robustly when the Blumenthal-Getoor index of the L\'evy measure is not too close to 2 (in particular, if the pure-jump part has finite variation). It does less well if this is not the case, since then the discretisation parameter $h$ must be chosen small, which is expensive in terms of numerics (viz. the size of $\hat{Q}^h$).

\bibliographystyle{plain}
\bibliography{Biblio_density_approximations}

\begin{thebibliography}{10}

\bibitem{higham_exp_to_vector}
Awad~H. Al-Mohy and Nicholas~J. Higham.
\newblock {Computing the Action of the Matrix Exponential, with an Application
  to Exponential Integrators}.
\newblock {\em {SIAM J. Sci. Comput.}}, 33(2):488--511, 2011.
\newblock MATLAB codes at
  \url{http://www.maths.manchester.ac.uk/~higham/papers/matrix-functions.php}.

\bibitem{bally}
Vlad Bally and Denis Talay.
\newblock {The Law of the Euler Scheme for Stochastic Differential Equations:
  II. Convergence Rate of the Density}.
\newblock {\em Monte Carlo Methods and Applications}, 2(2):93--128, 2009.

\bibitem{bertoin}
J.~Bertoin.
\newblock {\em L{\'e}vy Processes}.
\newblock Cambridge Tracts in Mathematics. Cambridge University Press,
  Cambridge, 1996.

\bibitem{bottcher}
Bj{\"o}rn B{\"o}ttcher and Ren{\'e}~L. Schilling.
\newblock Approximation of {F}eller processes by {M}arkov chains with {L\'e}vy
  increments.
\newblock {\em Stochastics and Dynamics}, 9(1):71--80, 2009.

\bibitem{madan}
Peter~P. Carr, H\'{e}lyette Geman, Dilip~B. Madan, and Marc Yor.
\newblock The fine structure of asset returns: an empirical investigation.
\newblock {\em Journal of Business}, 75(2):305--332, 2002.

\bibitem{cohen}
Serge Cohen and Jan Rosi{\'n}ski.
\newblock {Gaussian approximation of multivariate L\'evy processes with
  applications to simulation of tempered and operator stable processes}.
\newblock {\em Bernoulli}, 13(1):195--210, 2007.

\bibitem{conttankov}
R.~Cont and P.~Tankov.
\newblock {\em {Financial Modelling with Jump Processes}}.
\newblock Chapman and Hall/CRC Financial Mathematics Series. Chapman \&
  Hall/CRC, 2004.

\bibitem{contvoltchkova}
Rama Cont and Ekaterina Voltchkova.
\newblock {A Finite Difference Scheme for Option Pricing in Jump Diffusion and
  Exponential L\'evy Models}.
\newblock {\em SIAM J. Numerical Analysis}, 43(4):1596--1626, 2005.

\bibitem{crosby}
John Crosby, Nolwenn~Le Saux, and Aleksandar Mijatovi{\'c}.
\newblock Approximating {L\'e}vy processes with a view to option pricing.
\newblock {\em International Journal of Theoretical and Applied Finance},
  13(1):63--91, 2010.

\bibitem{dudley}
R.~M. Dudley.
\newblock {\em {Real Analyis and Probability}}.
\newblock Cambridge studies in advanced mathematics. Cambridge University
  Press, Cambridge, 2004.

\bibitem{evard}
J.-Cl. Evard and F.~Jafari.
\newblock {A Complex Rolle's Theorem}.
\newblock {\em The American Mathematical Monthly}, 99(9):858--861, 1992.

\bibitem{lopez}
Jos{\'e}~E. Figueroa-L{\'o}pez.
\newblock {Approximations for the distributions of bounded variation L\'evy
  processes}.
\newblock {\em {Statistics \& Probability Letters}}, 80(23-24):1744--1757,
  2010.

\bibitem{filipovic}
Damir Filipovic, Eberhard Mayerhofer, and Paul Schneider.
\newblock {Density Approximations for Multivariate Affine Jump-Diffusion
  Processes}.
\newblock {\em {The Journal of econometrics}}, 2012, forthcoming.

\bibitem{garronimenaldi}
M.~G. Garroni and J.~L. Menaldi.
\newblock {\em {Green functions for second order parabolic integro-differential
  problems}}.
\newblock Pitman research notes in mathematics series. Longman Scientific \&
  Technical, 1992.

\bibitem{glasserman}
Paul Glasserman.
\newblock {\em {Monte Carlo Methods in Financial Engineering (Stochastic
  Modelling and Applied Probability) (v. 53)}}.
\newblock Springer, 1st edition, 2003.

\bibitem{higham_scaling}
Nicholas~J. Higham.
\newblock The scaling and squaring method for the matrix exponential revisited.
\newblock {\em {SIAM J. Matrix Anal. Appl.}}, 26(4):1179--1193, 2005.

\bibitem{jacod}
J.~Jacod and A.~N. Shiryaev.
\newblock {\em Limit theorems for stochastic processes (2nd edition)}.
\newblock Grundlehren der mathematischen Wissenschaften. Springer-Verlag,
  Berlin Heidelberg, 2003.

\bibitem{kiessling}
Jonas Kiessling and Ra{\'u}l Tempone.
\newblock Diffusion approximation of {L\'e}vy processes with a view towards
  finance.
\newblock {\em Monte Carlo Methods and Applications}, 17(1):11--45, 2011.

\bibitem{kloeden}
Peter~E. Kloeden and Eckhard Platen.
\newblock {\em Numerical solution of stochastic differential equations}.
\newblock Springer-Verlag, Berlin; New York, 1992.

\bibitem{knopova}
Viktorya Knopova and Ren\'e~L. Schilling.
\newblock {Transition density estimates for a class of L\'evy and L\'evy-type
  processes}.
\newblock {\em {J. Theor. Probab.}}, 25(1):144--170, 2012.

\bibitem{higa}
Arturo Kohatsu-Higa, Salvador Ortiz-Latorre, and Peter Tankov.
\newblock Optimal simulation schemes for {L\'e}vy driven stochastic
  differential equations.
\newblock 2012.
\newblock arXiv:1204.4877 [math.PR].

\bibitem{kyprianou}
Andreas~E. Kyprianou.
\newblock {\em {Introductory Lectures on Fluctuations of {L\'e}vy Processes
  with Applications}}.
\newblock Springer-Verlag, Berlin Heidelberg, 2006.

\bibitem{mcc}
Dilip~B. Madan, Peter Carr, and Eric~C. Chang.
\newblock {The Variance Gamma Process and Option Pricing}.
\newblock {\em {European Finance Review}}, 2:79--105, 1998.

\bibitem{mijatovic}
Aleksandar Mijatovi\'c.
\newblock Spectral properties of trinomial trees.
\newblock {\em Proceedings of the Royal Society of London. Series A.
  Mathematical, Physical and Engineering Sciences}, 463(2083):1681--1696, 2007.

\bibitem{moler}
Cleve Moler and Charles~Van Loan.
\newblock {Nineteen Dubious Ways to Compute the Exponential of a Matrix,
  Twenty-Five Years Later}.
\newblock {\em {SIAM Review}}, 45(1):3--49, 2003.

\bibitem{norris}
James~R. Norris.
\newblock {\em Markov chains}.
\newblock Cambridge series in statistical and probabilistic mathematics.
  Cambridge University Press, Cambridge, 1997.

\bibitem{poirot}
J\'er\'emy Poirot and Peter Tankov.
\newblock {Monte Carlo Option Pricing for Tempered Stable (CGMY) Processes}.
\newblock {\em Asia-Pacific Financial Markets}, 13(4):327--344, 2006.

\bibitem{rosinski}
Jan Rosi{\'n}ski.
\newblock Simulation of {L\'e}vy processes in \textit{Encyclopedia of
  Statistics in Quality and Reliability: Computationally Intensive Methods and
  Simulation}.
\newblock Wiley, 2008.

\bibitem{sato}
K.~I. Sato.
\newblock {\em {L\'e}vy Processes and Infinitely Divisible Distributions}.
\newblock Cambridge studies in advanced mathematics. Cambridge University
  Press, Cambridge, 1999.

\bibitem{sidje}
Roger~B. Sidje.
\newblock Expokit: a software package for computing matrix exponentials.
\newblock {\em ACM Trans. Math. Softw.}, 24(1):130--156, 1998.
\newblock MATLAB codes at \url{http://www.maths.uq.edu.au/expokit/}.

\bibitem{stewart}
Ian Stewart and David Tall.
\newblock {\em Complex analysis}.
\newblock Cambridge University Press, Cambridge, 1983.

\bibitem{szimayer}
Alex Szimayer and Ross~A. Maller.
\newblock Finite approximation schemes for {L\'e}vy processes, and their
  application to optimal stopping problems.
\newblock {\em Stochastic Processes and their Applications},
  117(10):1422--1447, 2007.

\bibitem{sztonyk}
Pawel Sztonyk.
\newblock {Transition density estimates for jump L\'evy processes}.
\newblock {\em {Stochastic Processes and their Applications}},
  121(6):1245--1265, 2011.

\bibitem{tanakahiga}
Hideyuki Tanaka and Arturo Kohatsu-Higa.
\newblock {An operator approach for Markov chain weak approximations with an
  application to infinite activity L\'evy driven SDEs}.
\newblock {\em {The Annals of Applied Porbability}}, 19(3):1026--1062, 2009.

\end{thebibliography}

\end{document}